\documentclass[10pt]{amsart}

\usepackage[all]{xy}                        %
\CompileMatrices                            
\UseTips                                    

\input xypic
\usepackage[bookmarks=true]{hyperref}       

\usepackage{amssymb,latexsym,amsmath,amscd}
\usepackage{xspace}
\usepackage{enumerate}
\usepackage{graphicx}
\usepackage{vmargin}
\usepackage{todonotes}

\usepackage{dcpic,pictexwd}
\usepackage{tcolorbox}
\usepackage{subcaption}

\DeclareMathOperator{\rk}{rank}

\theoremstyle{plain}
\newtheorem{theorem}{Theorem}[section]
\newtheorem*{theorem*}{Theorem}
\newtheorem{proposition}[theorem]{Proposition}
\newtheorem{corollary}[theorem]{Corollary}
\newtheorem{lemma}[theorem]{Lemma}

\theoremstyle{definition}
\newtheorem{definition}[theorem]{Definition}

\newtheorem{remark}[theorem]{Remark}
\newtheorem{example}[theorem]{Example}

\newcommand{\enm}[1]{\ensuremath{#1}}          %

\newcommand{\cal}[1]{\mathcal{#1}}

\renewcommand{\bar}[1]{\overline{#1}}

\newcommand{\CC}{\enm{\mathbb{C}}}

\newcommand{\NN}{\enm{\mathbb{N}}}
\newcommand{\RR}{\enm{\mathbb{R}}}

\newcommand{\ZZ}{\enm{\mathbb{Z}}}
\newcommand{\FF}{{\enm{\mathbb{F}}}}

\newcommand{\JJ}{\enm{\mathbb{J}}}

\newcommand{\Cc}{\enm{\cal{C}}}
\newcommand{\Dd}{\enm{\cal{D}}}

\newcommand{\Ll}{\enm{\cal{L}}}

\newcommand{\Oo}{\enm{\cal{O}}}
\newcommand{\Pp}{\enm{\cal{P}}}

\newcommand{\Rr}{\enm{\cal{R}}}
\newcommand{\Ss}{\enm{\cal{S}}}

\newcommand{\Vv}{\enm{\cal{V}}}

\DeclareMathOperator{\bdeg}{\mathrm{bdeg}}

\renewcommand{\phi}{\varphi}
\renewcommand{\theta}{\vartheta}
\renewcommand{\epsilon}{\varepsilon}

\def\8{\infty}
\def\OF{\Oo_\FF}
\def\Of{\Oo}
\def\TT{T}
\def\es{\varnothing}
\def\ge{\geqslant}
\def\le{\leqslant}
\def\cd{\cdot}
\def\si{\sigma}
\def\De{\Delta}
\def\PP{\mathbb{P}}
\def\CP{{\PP^2}}
\def\CPv{{\PP^{2\vee}}}
\def\CPCPv{{\CP\times\CPv}}
\def\Gr{\mathbb{G}\mathrm{r}}
\def\PIqm{\pi_1^{-1}(q)\cap\pi_2^{-1}(m)}
\def\PUqm{\pi_1^{-1}(q)\cup\pi_2^{-1}(m)}
\def\gFS{g_\mathrm{FS}}
\def\Re{\mathop{\mathrm{Re}}}
\def\Im{\mathop{\mathrm{Im}}}

\usepackage{color}

\def\ds{\displaystyle}
\def\Hm{H_m}
\def\qH{{}_qH}

\def\la{\lambda}

\def\PGL{\mathop{\mathit{PGL}}}

\def\SL{\mathop{\mathit{SL}}}
\def\SU{\mathop{\mathit{SU}}}
\def\sm{\setminus}
\def\ts{\textstyle}
\def\vs{\vskip10pt}
\def\qbox#1{\quad\hbox{#1}\quad}
\def\lra{\longrightarrow}

\let\oldtocsubsection=\tocsubsection
\renewcommand{\tocsubsection}[2]{\hskip15pt\oldtocsubsection{#1}{#2}}

\begin{document}
\parskip1pt

\title[Twistor geometry of the Flag manifold]{\Large Twistor geometry of the Flag manifold}

\author[A. Altavilla]{Amedeo Altavilla}\address{Dipartimento di Matematica,
  Universit\`a degli Studi di Bari `Aldo Moro', via Edoardo Orabona, 4, 70125,
  Bari, Italia}\email{amedeo.altavilla@uniba.it}

\author[E. Ballico]{Edoardo Ballico}\address{Dipartimento Di Matematica,
  Universit\`a di Trento, Via Sommarive 14, 38123, Povo, Trento, Italia}
\email{edoardo.ballico@unitn.it}

\author[M. C. Brambilla]{Maria Chiara Brambilla}\address{
Universit\`a Politecnica delle Marche, via Brecce Bianche, I-60131 Ancona, Italia}
\email{brambilla@dipmat.univpm.it}

\author[S. Salamon]{Simon Salamon}\address{Department of Mathematics, King's
  College London, Strand, London, WC2R~2LS, UK} \email{simon.salamon@kcl.ac.uk}

\thanks{{\it The authors were partially supported as follows:}\\ \indent AA
  by GNSAGA and the INdAM project `Teoria delle funzioni ipercomplesse
  e applicazioni',\\ \indent MCB by GNSAGA and by the PRIN project `Geometria delle
  variet\`a algebriche',\\ \indent SS by the Simons Foundation
  (\#488635, Simon Salamon)}

\date{\today}

\subjclass[2010]{Primary: 32L25, 14M15; Secondary: 53C15, 53C28, 14J10, 15A21}
\keywords{Twistor space, flag manifold, del Pezzo surface, unitary equivalence}

\begin{abstract} 
A study is made of algebraic curves and surfaces in the flag manifold
$\mathbb{F}=SU(3)/T^2$, and their configuration relative to the
twistor projection $\pi$ from $\mathbb{F}$ to the complex projective
plane $\mathbb{P}^{2}$, defined with the help of an anti-holomorphic
involution $j$. This is motivated by analogous studies of algebraic
surfaces of low degree in the twistor space $\mathbb{P}^3$ of the
4-dimensional sphere $S^4$. Deformations of twistor fibers project to
real surfaces in $\mathbb{P}^{2}$, whose metric geometry is
investigated. Attention is then focussed on toric del Pezzo surfaces
that are the simplest type of surfaces in $\mathbb{F}$ of bidegree
$(1,1)$. These surfaces define orthogonal complex structures on
specified dense open subsets of $\mathbb{P}^{2}$ relative to its
Fubini-Study metric. The discriminant loci of various surfaces of
bidegree $(1,1)$ are determined, and bounds given on the number of
twistor fibers that are contained in more general algebraic surfaces
in $\mathbb{F}$.
\end{abstract}

\maketitle
\setcounter{tocdepth}{1} 
\tableofcontents

\section{Introduction}

The purpose of this paper is to study the complex 3-dimensional flag manifold
$\FF$ and some of the associated geometrical structures arising from its
description as the homogeneous space $SU(3)/T^2$. If we fix an invariant
complex structure on $\FF$ then there are three natural projections from $\FF$
to the complex projective plane $\CP$, one of which (call it $\pi$) is neither
holomorphic nor anti-holomorphic. The resulting three fibrations play an
implicit role in the classification of harmonic maps of surfaces into the
complex projective planes $\CP$ \cite{es,ew}, though in this paper, we shall
be more concerned with real branched coverings of $\CP$ defined by the choice
of an algebraic surface in $\FF$.

Let $p$ be a point of $\CP$ and $\ell$ a line in $\CP$. The pair $(p,\ell)$
defines a point of $\FF$ if $p\in\ell$. We can regard a line $\ell$ in $\CP$ as
a point in the dual complex projective plane $\CPv$, so that $\FF$ is naturally
an algebraic subvariety of $\CPCPv$. This is the standpoint that we adopt in
the early sections of this paper, in which our notation exploits to a maximum
the underlying elementary linear algebra. For example, the relation $p\in\ell$
is equivalent to the vanishing of the pairing $p\ell$, and the line through two
distinct points $p,q$ can be represented by a cross product $p\times
q$. Section \ref{prelim} uses the resulting double fibration to compute Hodge
numbers defined by line bundles $\Oo(a,b)$. This also enables us to associate a
bidegree to both curves and surfaces in $\FF$.

Section \ref{c&s} focusses attention on the most basic families of curves and
surfaces in the flag variety $\FF$. A family $\Vv$ of curves $L_{q,m}$ of
bidegree $(1,1)$ is parametrized by the complement of $\FF$ in $\CPCPv$, and
realizes each element of $\Vv$ as the intersection $\qH\cap\Hm$ of two
Hirzebruch surfaces of type $1$. Each of these surfaces can, merely by their
description as a subvariety of $\FF$, be viewed simultaneously as a $\PP^1$
bundle over $\PP^1$ and as the blowup of $\CP$ at one point. An arbitrary
smooth curve of bidegree $(1,1)$ has the form $L_{q,m}$ for some $(q,m)$ with
$qm\ne0$. We study intersections between members of this family and various
Hirzebruch surfaces.
    
Section \ref{Kronecker} deals with the classification of surfaces $S$ in $\FF$
of bidegree $(1,1)$, each of which corresponds to a complex $3\times3$ matrix
$A$ up to the addition of a scalar multiple of the identity and
rescaling. There is an analogy with the simultaneous diagonalization of
quadratic forms, but in general $A$ will not be diagonalizable, which leads to
singular and reducible examples. Indeed, the classification of all such
surfaces provides a geometrical illustration of Jordan canonical form in the
simplest of cases. We then proceed to study the equivalence of such surfaces
under unitary transformations.

We begin to examine the twistor picture in Section \ref{Twy}. The
Hermitian structure allows one to associate to $p$ a line
$p^*\in\CPv$, and to $\ell$ a point $\ell^*\in\CP$. The anti-linear
involution $j\colon(p,\ell)\mapsto (\ell^*,p^*)$ of $\FF$ has no fixed
points. We can then define $\pi$ by mapping $(p,\ell)$ to the point of
$\CP$ determined by $p^*\times\ell$. Then $\pi$ commutes with $j$, and
exhibits $\FF$ as the twistor space of the complex projective plane
$\CP$ with its standard (self-dual) Fubini-Study metric. The
distinction between $\CP$ and $\CPv$ is now less important, and we
obtain a triple $\pi_1,\pi,\pi_2^*$ of fibrations of $\FF\to\CP$. They
are permuted by means of outer automorphisms (of which $j$ is one)
arising from the Weyl group of $\SU(3)$, but are distinguished by our
choice of complex structure on $\FF$.

Local sections $s$ of $\pi$ parametrize almost complex structures $J$ on open
sets of $\CP$, and a fundamental property of a twistor space asserts that the
image of $s$ is holomorphic if and only if $J$ is complex, i.e., its Nijenhuis
tensor vanishes. The involution $j$ maps $J$ to $-J$ and, in the twistor
context, $j$-invariant objects are called `real'. In the analogous situation of
the Penrose fibration $\PP^3\to S^4$, there has been extensive study of
algebraic surfaces in $\PP^3$ and their associated orthogonal complex
structures in domains of $S^4$ \cite{sv1,altavillaballico1, altavillaballico2, altavillaballico3, altavillasarfatti, chirka, armstrong, APS, gensalsto, sv2}.

Section \ref{Twn} explains the relevance of the basic geometry of curves and
surfaces to the twistor theory. The fibers of the twistor fibration $\pi$ form
a real subfamily of $\Vv$, whereas a generic smooth curve $L_{q,m}$ projects to
a surface of revolution, whose first fundamental form (induced from the
Fubini-Study metric) we identify. As $q$ approaches the line $m$, the image
predictably acquires a dumbbell shape, reflecting the degeneration of $L_{q,m}$
to two lines. The underlying $U(1)$ symmetry enables us to visualize this in
Figure \ref{rot}, and other relevant surfaces of revolution are constructed in
Section \ref{ramf} and displayed in Figure \ref{disc}.

If a complex $3\times3$ matrix $A$ has three distinct eigenvalues, then the
associated surface $S$ in $\FF$ is a smooth del Pezzo surface of degree 6. Any
such surface is invariant by the action of the maximal torus $T^2$ of diagonal
matrices in $SU(3)$, and this allows us to use toric methods to describe its
behaviour relative to the twistor projection. Such surfaces $S$ have bidegree
$(1,1)$, and are the analogues of quadrics in $\PP^3$. For example, $S$ is
$j$-invariant if and only if $A$ is Hermitian, and in this case $S$ contains a
family of twistor fibers parametrized by a circle. It then becomes a natural
problem to understand the configuration of such a surface relative to $\pi$,
and to determine how many fibers of $\pi$ it can contain.

The problem of determining the branch locus of certain (real or toric) surfaces
of bidegree $(1,1)$ relative to $\pi$ is considered in Section \ref{Tw1}. In
Section \ref{infinity}, we learn that a $(1,1)$ surface can contain zero, one,
two, or (if real) infinitely many twistor fibers. This is first proved by
B\'ezout-type methods, and then more explicitly. All these cases are realized
by various examples. The final Section \ref{ramf} describes the twistor fibers
and branch loci of smooth but non-real $(1,1)$ surfaces.\smallskip

We conclude with some observations that will not be pursued in this
paper, but which suggest alternative approaches to, and
generalizations of, our work.

There are close analogues of our results with those of \cite{sv1} on
the Penrose fibration $\PP^3\to S^4$. This is to be expected since the
twistor spaces $\PP^3$ and $\FF$ incorporate an open orbit of a
complex Heisenberg group, and are birationally equivalent, a fact that
extends to any two Wolf spaces of the same dimension \cite{Bur}. From
the twistor viewpoint, surfaces of bidegree $(1,1)$ in $\FF$ evidently
correspond to quadrics in $\PP^3$; this is the conclusion of the
investigation in Section \ref{ramf}. We expect surfaces of bidegree
$(2,1)$ or $(1,2)$ (`del Pezzo double planes') to relate to cubic
surfaces in $\PP^3$. The latter can contain at most 5 twistor lines
(and some do) \cite{sv2}, and according Corollary
\ref{number-twistorlines} the former can contain at most 6 twistor
lines (but this may not be optimal). On the other hand, the flag
variety is a richer environment in which to study surfaces, owing to
its natural fibrations to planes. For example, surfaces of bidegree
$(2,2)$ in $\FF$ are K3 surfaces for which the non-commuting
involutions arising from $\pi_1$ and $\pi_2$ give rise to non-trivial
dynamics \cite{Weh}.

The sphere $S^4$ admits an infinite series of $SO(3)$-invariant
self-dual Einstein metrics $g_k$ ($k\ge3$) on $S^4$ with an orbifold
singularity with cone angle $2\pi/(k-2)$ along an embedded Veronese
surface. These metrics, along with their twistor spaces $Z_k$, were
described by Hitchin \cite{hitchin4}. In the case $k=4$, the orbifold
is the global quotient of $\CP$ (equipped with its Fubini-Study
metric) by complex conjugation, and $Z_4$ is a secant variety in
$\PP^4$, see Remark \ref{secant}.

For a generalization of the approach of Section \ref{Twy} and triple
fibrations in the context of $\mathit{Spin}(7)$ and triality, we cite
\cite{MW}.

\section{Preliminaries}\label{prelim}

In this section, we set up notation that will allow us to work with the flag
manifold $\FF$.

Throughout the paper, we denote by $\PP^n$ the complex projective space
$\PP(\CC^{n+1})$, and by $\PP^{n\vee}$ its dual $\PP((\CC^{n+1})^\vee)$. For
the most part, $n$ will equal $2$. We shall always regard elements of $\CC^3$
as row vectors, and elements of the dual space $(\CC^3)^\vee$ as column
vectors, with transpose also indicated by the superscript ${}^\vee$. We shall
retain this distinction at the level of homogeneous coordinates, so that the
row vector $p=(p_0,p_1,p_2)$ defines $[p_0:p_1:p_2]$ in $\CP$, and the column
vector $\ell=(\ell_0,\ell_1,\ell_2)^\vee$ defines $[\ell_0:\ell_1:\ell_2]^\vee$
in $\CPv$. Assuming our vectors are non-zero, we shall abuse notation by
writing $p\in\CP$ and $\ell\in\CPv$, so that the assertions $p\in\ell$
(geometry) and $p\ell=0$ (algebra) can be used interchangeably.

We study now the bi-projective space $\CPCPv$. Its Segre embedding into $\PP^8$
is induced by the map $(p,\ell)\mapsto \ell p$, in homogeneous coordinates
\[\big(\,[p_0:p_1:p_2],\>[\ell_0:\ell_1:\ell_2]^\vee\big) \longmapsto
\left[\begin{matrix} 
p_0\ell_0 & p_1\ell_0 &p_2\ell_0\\
p_0\ell_1 & p_1\ell_1 &p_2\ell_1\\
p_0\ell_2 & p_1\ell_2 &p_2\ell_2\\
\end{matrix}\right].\]
Its image (the Segre variety) is a fourfold of degree $6$ in $\PP^8$, see
for example \cite{h}. We have a diagram

\[
\begindc{\commdiag}[3]
\obj(0,0)[A]{$\CP$}
\obj(260,0)[Aa]{$\CPv$}
\obj(130,130)[B]{$\CPCPv$}
\mor{B}{A}{$\Pi_1$}[\atright,\solidarrow]
\mor{B}{Aa}{$\Pi_2$}[\atleft,\solidarrow]
\enddc\]\vs

\noindent in which $\Pi_1$ and $\Pi_2$ are the standard projections with
$\Pi_1(p,\ell)=p$ and $\Pi_2(p,\ell)=\ell$. Their fibers are linear sections of
the Segre variety of codimension $2$.

What follows is some relevant algebra. Let $\Rr:=\CC[p_0, p_1, p_2]$ be the
complex vector space of all homogeneous polynomials in the variables $p_0, p_1,
p_2$. Analogously, set $\Rr^{\vee}:=\CC[\ell_0, \ell_1, \ell_2]$.  The spaces
$\Rr$ and $\Rr^{\vee}$ are graded in the usual way, so that $\Rr_{a}$ denotes
the vector space of homogeneous polynomials of degree $a$ in the variables
$p_0, p_1, p_2$.
In view of the Segre embedding, we consider
$\Pp:=\Rr\otimes_\CC\Rr^{\vee}$. Then $\Pp$ is a polynomial ring in the
variables $p_i,\ell_j$, bigraded in the following way. Set $\Pp_{a,b}:=
\Rr_a\otimes_\CC\Rr^\vee_b$ for $(a,b)\in\NN^2$, so that
\begin{equation}\label{Pp}\ts
  \Pp= \bigoplus\limits_{(a,b)\in \NN^2} \Pp_{a,b},\quad \dim \Pp_{a,b}
  =\binom{a+2}2\binom{b+2}2
\end{equation}
Multiplication of polynomials induces a bilinear map $\Pp_{a,b}\times \Pp_{c,d}
\to \Pp_{a+c,b+d}$.

For any $a,b\in\ZZ$, we set $\Oo_\CPCPv(a,b)=\Pi_1^*\Oo_\CP(a)\otimes
\Pi_2^*\Oo_\CPv(b)$. The Leray-Hirsch theorem (or the K\"unneth formula)
implies that $\mathrm{Pic}(\CPCPv)\cong \ZZ^2$ is freely generated as an
abelian group by $\Oo_\CPCPv(1,0)$ and $\Oo_\CPCPv(0,1)$. Since the canonical
line bundle $\omega_\CP$ is isomorphic to $\Oo_\CP(-3)$, we have the expression
\[ \omega_\CPCPv\cong \Oo_\CPCPv(-3,-3).\]
for the canonical bundle of $\CPCPv$.

Next, we compute Hodge numbers using the K\"unneth formula. This gives
\begin{equation}\label{eqp1}
h^i(\Oo_\CPCPv(a,b))=\sum_{j=0}^{i} h^j(\Oo_\CP(a))\,h^{i-j}(\Oo_{\CPv}(b))
\end{equation}
for all $i\in \NN$ and $(a,b)\in \ZZ^2$. From (\ref{eqp1}), Serre duality and
the cohomology of line bundles on $\CP$, we deduce

\begin{lemma}\label{p1}
\begin{enumerate}
\item[(i)] 
$\ds h^0(\Oo_\CPCPv(a,b))=\begin{cases}
\binom{a+2}{2}\binom{b+2}{2}&\mbox{if}\quad(a,b)\in \NN^2\\
\;0&\mbox{if either $a<0$ or $b<0$};
\end{cases}$
\item[(ii)] $h^1(\Oo_\CPCPv(a,b))=0$ for all $(a,b)\in \ZZ^2$;
\item[(iii)] $h^2(\Oo_\CPCPv(a,b))=0$ if either $a\ge 0$ and $b\ge -2$,
  or $a=-1$, or $a\le -2$ and
$b<0$;
\item[(iv)] $h^3(\Oo_\CPCPv(a,b))=0$ for all $(a,b)\in \ZZ^2$;
\item[(v)] $h^4(\Oo_\CPCPv(a,b))=0$ if either $a\ge -2$, or $b\ge -2$.
\end{enumerate}
\end{lemma}

These results will be refined in the next subsection.

\subsection{The flag manifold}
We next define the main object of study.

\begin{definition}
The \textit{flag manifold} is the algebraic subvariety of $\CPCPv$ given by 
\[\FF:=\big\{(p,\ell)\in\CPCPv\;|\;p\in\ell\big\}.\]
\end{definition}

\noindent Since the condition $p\in\ell$ is equivalent to $p\ell=0$, we
have the coordinate description
\[ \FF:=\{([p_0:p_1:p_2],\>[\ell_0:\ell_1:\ell_2]^\vee)\in\CPCPv\;\big|\;
p_0\ell_0+p_1\ell_1+p_2\ell_2=0\}.\] In future, we shall favour the algebraic
way of expressing incidence. We shall denote the restrictions of the standard
projections to the flag manifold in lower case: $\pi_i:=\Pi_i|_\FF$ for
$i=1,2$. The two maps $\pi_1:\FF\to\CP$ and $\pi_2:\FF\to\CPv$ are locally
trivial $\PP^1$-bundles. In particular, $\FF=\PP(\Omega^1_\CP(1))$ with
$\pi_1:\PP(\Omega^1_\CP(1))\to\CP$ the natural projection as a $\PP^1$-bundle.
 
The fibers of $\pi_1$ and $\pi_2$ can easily be described explicitly. Let
$q\in\CP$ and $m\in\CPv$. Then
\begin{align*}
  \pi_1^{-1}(q) &= \{(q,\ell)\in \FF\mid q\ell=0\},\\
  \pi_2^{-1}(m) &= \{(p,m)\in \FF\mid pm=0\}
\end{align*}
are linear sections of codimension $3$ and so smooth rational curves. Observe
that $\PIqm\ne\es$ if and only if $qm=0$.

With notation from the previous subsection, we can regard $p\ell =
p_0\ell_0+p_1\ell_1+p_2\ell_2$ as an element of $\Pp_{1,1}$. Set
$\Ss:=\Pp/(p\ell)$. Since $\Ss$ is the quotient of a bigraded polynomial ring
by a principal ideal generated by a bi-homogeneous polynomial,
\begin{equation}\label{Ss}\ts
  \Ss = \bigoplus\limits_{a,b} \Ss_{a,b}
\end{equation}
is also bi-graded. Here, $\Ss_{a,b}$ is a complex vector space, and
multiplication in the ring $\Ss$ induces a bilinear map $\Ss_{a,b}\times
\Ss_{c,d} \to \Ss_{a+c,b+d}$.

Set
\[\OF(a,b):= \pi_1^*\Oo_\CP(a)\otimes\pi_2^*\Oo_\CPv(b).\]
The Leray-Hirsch theorem applied to the map
$\pi_1:\PP(\Omega^1_\CPv(1)) \to \CPv$ implies that any line bundle on
$\FF$ is isomorphic to $\OF(a,b)$ for a unique $(a,b)\in\ZZ^2$.  Since
$\FF$ is an effective divisor of $\CPCPv$, we can apply this statement
to the line bundle that generates $\FF$. Indeed, since the fibers of
$\Pi_i$ are linear sections of the Segre embedding of $\CPCPv$, we
must have $\FF\in |\Oo_\CPCPv(1,1)|$.

Since $\omega_\CPCPv\cong\Oo_\CPCPv(-3,-3)$, the adjunction formula gives the
expression
\begin{equation}\label{can}
  \omega_\FF\cong\OF(-2,-2)
\end{equation}
for the canonical bundle of the flag manifold.
 
From the exact sequence
\begin{equation}\label{eqp2}
0 \to\Oo_\CPCPv(a-1,b-1)\to \Oo_\CPCPv(a,b)\to\OF(a,b)\to 0
\end{equation}
and Lemma \ref{p1} we get:

\begin{lemma}\label{p2}
\begin{enumerate}
\item[(i)]
$\ds h^0(\OF(a,b))=\begin{cases}
\binom{a+2}{2}\binom{b+2}{2} - \binom{a+1}{2}\binom{b+1}{2}
&\mbox{if}\quad(a,b)\in \NN^2\\
\;0&\mbox{if either $a<0$ or $b<0$}
\end{cases}$
\item[(ii)] $h^1(\OF(a,b)) =0$ if either $a\ge 0$ and $b\ge 0$,
  or $a\le 0$ and $b\le 0$;
\item[(iii)] $h^2(\OF(a,b)) =0$ if either $a\ge 0$ and $b\ge -2$,
  or $a=-1$, or $a\le -2$ and $b<0$.
\end{enumerate}
\end{lemma}

\noindent Since $\Ss_{a,b} = H^0(\OF(a,b))$, we have that $\Ss_{a,b}$ has
dimension $\binom{a+2}{2}\binom{b+2}{2}-\binom{a+1}{2}\binom{b+1}{2}$, for all
$(a,b)\in \NN^2$. In particular $\dim(\Ss_{1,1})=8$.\smallbreak

Basic results regarding the flag manifold from a related point of view
can also be found in \cite[\S1.1]{malaspina}.
  
\subsection{Automorphisms}\label{Aut}
This subsection briefly describes the projective and unitary automorphisms of
the flag manifold $\FF$. We shall only need the latter from Subsection \ref{U3}
onwards, so we start from the ambient bi-projective space.

The family of automorphisms of $\CPCPv$ is generated by pairs $(B_1,B_2)$ in
$\SL(3,\CC)\times\SL(3,\CC)$, acting via matrix multiplication as
\[ (B_1,B_2)\cdot(p,\ell)=(pB_1^\vee,\>B_2\ell),\]
together with the involution
\begin{equation}\label{kappa}
  \kappa\colon(p,\ell)\longmapsto(\ell^\vee,p^\vee),
\end{equation}
in accordance with \cite{tango}. To obtain the family of
automorphisms of $\FF$ it is sufficient to consider those
transformations that preserve the equation $p\ell=0$. Applying
$(B_1,B_2)$, one gets
\[ p\ell = (pB_1^\vee)(B_2\ell) = p(B_1^\vee B_2)\ell,\]
so $B_1=(B_2^{-1})^\vee$. We deduce that the automorphisms of $\FF$ are generated by
matrices $B\in\SL(3,\CC)$ acting as
\begin{equation}\label{PGL}
  B\cdot(p,\ell)=(pB^{-1},B\ell),
\end{equation}
together with $\kappa$. Of course, $e^{2\pi i/3}I$ acts trivially on
$\CC^3$, so the connected Lie group acting effectively is
$\PGL(3,\CC)=\SL(3,\CC)/\ZZ_3$.

The subgroup of projective unitary transformations is defined by imposing a
reduction from $SL(3,\CC)$ to the special unitary group $\SU(3)$. For the
purposes of calculation, it will be more convenient to allow $B\in U(3)$, given
that the centre of $U(3)$ will always act trivially. The realization of the
flag manifold as the twistor space $\FF$ of $\CP$ will require us to restrict
to this compact subgroup.

Having fixed an origin in $\FF$, one can further reduce $\SU(3)$ to
its standard maximal torus $\TT$ consisting of diagonal
matrices. Because the resulting isotropy representation of
$\FF=\SU(3)/\TT$ has three irreducible real components
\cite[\S12]{es}, there is (up to homothety) a 2-parameter family of
$\SU(3)$-invariant Riemannian metrics on $\FF$. This family includes
two Einstein metrics, a K\"ahler and a nearly-K\"ahler one. Both can
be constructed as submersions over $\CP$ endowed with the Fubini-Study
metric $\gFS$, see for example \cite{Musk}.

The choice of the maximal torus $\TT$ gives rise to the Weyl group
$W=N(\TT)/\TT\cong S_3$, where $N(\TT)$ is the normalizer of $\TT$ in $SU(3)$.
An element of $W$ can be represented by a matrix in $\SU(3)$ permuting the
coordinates. This projective action will be relevant in the classification of
canonical forms for surfaces, see Remark \ref{Xratio}. There is however a
different representation of $W$ that acts non-trivially on cohomology and is
especially relevant to the twistor geometry of Section \ref{Twy}, see Remark
\ref{JJ}.

\section{Some curves and surfaces in the flag manifold}\label{c&s}
          
We start with the concept of bidegree for a curve. An algebraic
submanifold is said to be \textit{integral} if it is reduced and
irreducible. The symbol $\simeq$ will be used throughout this paper to
indicate biholomorphism.

\begin{definition}
Let $C\subset \FF$ be an integral curve. We define its bidegree
$\bdeg(C)=(d_1,d_2)$ as follows: we say that $d_i=0$ if $\pi_i(C)$ is a point;
otherwise $d_i=a_ib_i$, where $a_i=\deg(\pi_i(C))$ and $b_i=\deg(\pi|_{C})$.
\end{definition}

\noindent If a curve $D$ has irreducible components $C_1,\ldots,C_s$
then the bidegree $\bdeg(D)$ is taken to be the sum of the bidegrees
$\bdeg(C_1),\ldots,\bdeg(C_s)$.\smallbreak

The fibers of $\pi_1$ and $\pi_2$ provide the most obvious examples. For any
$p\in\CP$ and $m\in\CPv$, we have
\[\begin{array}{l}
\bdeg(\pi_1^{-1}(p))=(0,1),\\
\bdeg(\pi_2^{-1}(m))=(1,0).
\end{array}\]
Moreover, we can identify
\begin{equation}\label{rk1001}
N_{\pi_1^{-1}(p),\FF}=\pi_1^{*}(N_{p,\CP})=\pi_1^{*}(\Oo_\CP\oplus\Oo_\CP) =
\OF\oplus\OF.
\end{equation}
as the normal bundle of a fiber of $\pi_1$ in $\FF$.

\begin{remark}\label{rk2001}
  Let $C\subset \FF$ be an integral projective curve with $\bdeg(C)=(d_1,d_2)$.
  Then, for $i\in\{1,2\}$, we have the following:
\begin{itemize}
\item
If $d_i=1$, then $C$ is rational. Indeed, since $C$ is integral, the
map $\pi_{i|C}\colon C\to\PP^1$ is birational. Let $u\colon C'\to C$
be the normalization map, so that $\pi_{i|C}\circ u\colon C'\to\PP^1$
is a degree-one morphism between smooth curves. Thus $\pi_{i|C}\circ u$
and $\pi_{i|C}$ are isomorphisms, and $\deg(\pi_i(C))=1$.
\item If $d_i=0$, then $d_{3-i}=1$. Indeed, if $\pi_i(C)=\{p\}$, then
  $C\subseteq \pi_i^{-1}(p)\simeq \PP^1$ and hence, since $C$ is
  integral, $C= \pi_i^{-1}(p)$. In particular, by \eqref{rk1001},
  $N_{C,\FF}=\Oo^2_\FF$.
\end{itemize}
This concludes an analysis of the easy cases.
\end{remark}

\subsection{Curves of bidegree (1,1)}\label{Lqm}
We next define the fundamental family of curves of bidegree $(1,1)$.
See~\cite[p. 438, Example 3]{ahs}, \cite[p. 147]{hitchin},
\cite[\S4.3]{gauduchon}  and~\cite[Section 1.1]{malaspina}.

\begin{definition}\label{def rette}
  Fix $(q,m)\in\CPCPv$ such that $(q,m)\in(\CPCPv)\sm\FF$, so $qm\ne0$.
The formula
\[ L_{q,m}:=\{(p,\ell)\in\FF\mid p\in m,\ \ell\ni q\} = 
\{(p,\ell)\in\CPCPv\mid  p\ell=0,\ q\ell=0,\ pm=0\}\]
defines a family $\Vv$ of curves in $\FF$.
\end{definition}

\begin{figure}
\vspace{-15pt}
\includegraphics[scale=0.7]{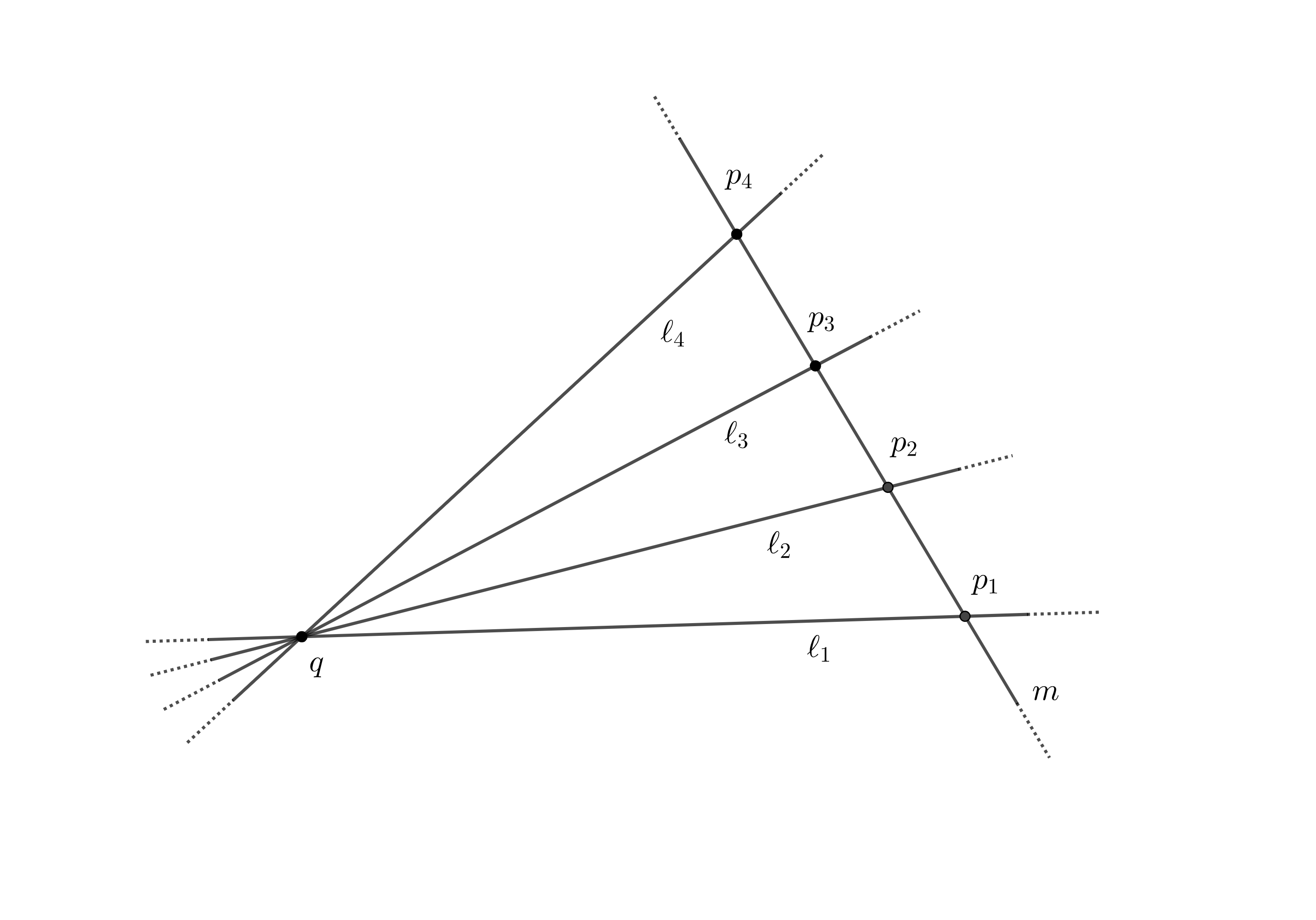}
\vspace{-25pt}
\caption{Suppose that $q\notin m$. Then $L_{q,m}$ consists of pairs $(p,\ell)$
  such that $q\in\ell$ and $p=\ell\cap m$. In the figure, $(p_i,\ell_i)$ are
four elements of $L_{q,m}$.}
\end{figure}

Given $p,q\in\CP$, the line passing through $p$ and $q$ is represented by the column
vector
\[ (p_1 q_2-p_2 q_1:p_2 q_0-p_0 q_2:p_0 q_1-p_1 q_0)^\vee, \]
which we shall denote by $p\times q$. The cross product is the natural
isomorphism
\begin{equation}\label{cross}
  \Lambda^2(\CC^3)\stackrel\cong\lra(\CC^3)^\vee,
  \end{equation}
induced by an $\SL(3,\CC)$ structure on $\CC^3$. In the same way, the
intersection $\ell\cap m$ of two lines is the point represented by the row
vector
\[ (\ell_1 m_2-\ell_2 m_1,\>\ell_2 m_0-\ell_0 m_2,\>\ell_0 m_1-\ell_1 m_0),\]
which we denote by $\ell\times m$. We further abuse notation by writing
$p\times q\in\CPv$ and $\ell\times m\in\CP$.

The cross product formulae are suggestive of computation, and we shall use them
below in preference to the equivalent set theoretic statements.

\begin{remark}\label{intersectwithfbire}
It is easy to see that $\bdeg(L_{q,m})=(1,1)$. Indeed, $L_{q,m}\cap
\pi_1^{-1}(q')\ne\es$ if and only if $q'm=0$, so $\pi_1(L_{q,m})=m$ and 
$$L_{q,m}\cap \pi_1^{-1}(q') = \{(q',q\times q')\}.$$
Note that $q\times q'$ is defined because $q$ does not lie on $m$. Similarly,
$L_{q,m}\cap \pi_2^{-1}(m')\ne\es$ if and only if $qm'=0$, in which case
$$L_{q,m}\cap \pi_2^{-1}(m')=\{(m\times m',m')\}.$$
By Remark~\ref{rk2001}, we have $L_{q,m}\simeq\PP^1$ for any $(q,m)\in
(\CPCPv)\sm\FF$.

The geometry of such curves is described further in Lemma~\ref{tutte11} below.
\end{remark}

Our next result describes the possible intersections of two elements in $\Vv$.

\begin{lemma}\label{lemmaintersection11}
  Let $(q,m), (q',m')$ be distinct points of $(\CPCPv)\sm\FF$.
\begin{enumerate}
\item[(i)] If $q\ne q'$ and $m\ne m'$, then $L_{q,m}\cap L_{q',m'}\ne \es$ if
  and only if the point $m\cap m'$ lies on the line $qq'$, in which case
  $L_{q,m}\cap L_{q',m'} = \{(m\times m',\>q\times q')\}$.
\item[(ii)] If $q= q'$ and $m\ne m'$, then $L_{q,m}\cap L_{q',m'} =
  \{(m\times m',q\times(m\times m'))\}$.
\item[(iii)] If $q\ne q'$ and $m= m'$, then $L_{q,m}\cap
  L_{q',m'} = \{((q\times q')\times m,\>q\times q')\}$.
\end{enumerate}
\end{lemma}

\begin{proof}
The results rely on an analysis of the following system, representing a point
$(p,\ell)$ of intersection of the $(1,1)$-curves $L_{q,m}$ and $L_{q',m'}$:
\begin{equation}\label{systemLqm}
\begin{cases}
\ p_0\ell_0+p_1\ell_1+p_2\ell_2=0\\[2pt]
\ q_0\ell_0+q_1\ell_1+q_2\ell_2=0\\
\ p_0m_0+p_1m_1+p_2m_2=0\\[2pt]
\ q_0'\ell_0+q_1'\ell_1+q_2'\ell_2=0\\
\ p_0m_0'+p_1m_1'+p_2m_2'=0
\end{cases}
\end{equation}
Assume first that $q\ne q'$ and $m\ne m'$. By considering the third and fifth
equations, we get that the first component $p$ of the intersection must lie on
$m$ and $m'$, i.e.\ $p=m\times m'$. Similarly, by from the second and fourth
equation, the second component $\ell$ equals $q\times q'$. Given the first
equation (characterizing $\FF$), this solution is admissible if and only
$(m\times m')(q\times q')=0$.

Assume next that $q=q'$ and $m\ne m'$. As before, we get that $p=m\times
m'$. The latter is distinct from $q$ (for otherwise $q$ would lie in $m\cap
m'$), hence we see that $\ell$ must be the line $pq$. If follows that
$\ell=q\times(m\times m')$, and (ii) is established.

Case (iii) is completely analogous.
\end{proof}

The following is an immediate consequence of the previous lemma:

\begin{corollary}\label{condition-intersection}
Let $(q,m),(q',m')\in(\CPCPv)\sm\FF$. Then $L_{q,m}$ intersects $L_{q',m'}$ if
and only if $(m\times m')(q\times q')=0$.
\end{corollary}

\begin{example}
Take $L_{q,m},\,L_{q',m'}\in\Vv$ such that
\[ q=[i:0:0],\quad m=[1:1:0],\qbox{and} q'=[1:1:0],\ m'=[i:0:0].\]
Then $L_{q,m}\cap L_{q',m'}=\es$. For $q\times q'=[0:0:i]^\vee$ and $m\cap
m'=[0:0:i]$.
\end{example}

We have seen that $(\CP\times\CPv)\sm \FF$ parametrizes a complex
$4$-dimensional family $\Vv$ of rational curves in $\FF$. If we extend the
definition of $L_{q,m}$ to the case $q\in m$, we find that
  \begin{equation}\label{Lsing}
    L_{q,m} = \PUqm
    \end{equation}
is the union of the respective fibers meeting in
$(q,m)\in\FF$. Referring to Definition \ref{def rette}, these two
fibers correspond to the respective possibilities that $p=q$ (which
forces $q\ell=0$) or $\ell=m$ (which forces $pm=0$). The algebraic
closure $\bar\Vv$ is therefore $\CPCPv$, formed by adjoining these
reducible curves. In other words, $\bar\Vv$ is the Hilbert scheme that
parametrizes all closed subschemes of $\FF$ having Hilbert polynomial
$2t+1$ with respect to the Segre embedding. See also~\cite[Lemma
  1.5]{malaspina}).

\subsection{Surfaces of bidegree (1,0) and (0,1)}
As for the case of curves, but perhaps more naturally, we can consider a notion
of bidegree. Subsequently, we shall focus mainly on the cases of `low'
bidegree.

Let $d_1,d_2\in\NN$ be any pair of natural numbers. We set
$\OF(d_1,d_2)=\pi_1^*\Oo_\CP(d_1)\otimes \pi_2^*\Oo_{\CPv}(d_2)$, and use
$|\OF(d_1,d_2)|$ to denote the projective space $\PP(H^0(\OF(d_1,d_2)))$. From
now on we denote $\OF$ by $\Of$ if no confusion will arise.

\begin{definition}
  Let $S\subset \FF$ be an algebraic surface in $|\Of(d_1,d_2)|$.  Then we
  say that $S$ has bidegree $(d_1,d_2)$, and we write $\bdeg(S)=(d_1,d_2)$.
\end{definition}

Fix $m\in\CPv$ and $q\in\CP$. The former represents a line in $\CP$, which
set-theoretically equals $\pi_1(\pi_2^{-1}(m))$. The latter defines a line in
$\CPv$ (corresponding to a pencil of lines in $\CP$), which is of course
$\pi_2(\pi_1^{-1}(q))$. These objects pull back to surfaces in $\FF$:

\begin{definition}\label{qHm}
Given $m\in\CPv$ and $q\in\CP$, set
\begin{align*}
  \Hm&:=\pi_1^{-1}(\pi_1(\pi_2^{-1}(m)))=\{(p,\ell)\in \FF\mid p\in m\}\\
  \qH&:=\pi_2^{-1}(\pi_2(\pi_1^{-1}(q)))=\{(p,\ell)\in \FF\mid \ell\ni q\}.
\end{align*}
\end{definition}

In words, $\Hm$ is the set of pairs $(p,\ell)\in \FF$ such that $p$ moves on
$m$, while $\qH$ is the set of pairs $(p,\ell)\in \FF$ such that $\ell$
contains $q$. Using the cross product, we also have
\begin{align}\label{Hcross}
  \Hm&:=\{(p,\ell)\mid p\ell=0,\ pm=0\}
  = \{(\ell\times m,\ell)\mid \ell\in\CPv\}\\
  \qH&:=\{(p,\ell)\mid p\ell=0,\ q\ell=0\}
  = \{(p,p\times q)\mid p\in\CP\}.
\end{align}

If $p,q$ are distinct points of $\CP$, then
\[ \pi_1^{-1}(p)\cap\qH = \{(p,p\times q)\}\]
consists of a single point. Similarly, if $l\ne m$ then
\[ \pi_2^{-1}(\ell)\cap\Hm = \{(\ell\times m,\ell)\}.\]
It follows that $\bdeg(\qH)=(0,1)$ and $\bdeg(\Hm)=(1,0)$. We can also deduce
this from \eqref{Pp}, since $\qH$ is defined by the equation $q\ell=0$ that is
linear (of degree one) in the $\ell_j$ and $\Hm$ is defined by one linear in
the $p_i$.

On the other hand, it is easy to see that any surface $S\in|\Of(0,1)|$ has
the form $\qH$ for some $q\in\CP$, and any surface $S\in|\Of(1,0)|$ has
the form $\Hm$ for some $m\in\CPv$. Hence the family of surfaces
$S\in|\Of(0,1)|$ is parametrized by $\CP$ and has complex dimension
$2$. Any two elements of the family are projectively equivalent because
$\SL(3,\CC)$ acts transitively on $\CP$. The same considerations hold for
elements in $|\Of(1,0)|$.

Since $\pi_{1}^{-1}(q)\subset \qH$ and $\pi_{2}^{-1}(m)\subset\Hm$, we obtain the
following result.

\begin{corollary}\label{fritz}
Both $\Hm$ and $\qH$ are Hirzebruch surfaces of type $1$. Indeed,
$\pi_1$ bestows upon $\qH$ the structure of a blow-up of $\CP$ at $q$,
and $\pi_2$ realizes $\Hm$ as a blow-up of $\CPv$ at $m$.
\end{corollary}

Any two Hirzebruch surfaces of `like' type intersect in a fiber of $\pi_1$ or
$\pi_2$. Indeed, for $m\ne m'$ and $q\ne q'$, we have
\[  \Hm\cap H_{m'}=\pi_1^{-1}(m\times m'),\qquad
\qH\cap {}_{q'}H=\pi_2^{-1}(q\times q').\] It follows that any two bidegree
$(1,0)$ or $(0,1)$ surfaces always meet, while the intersection three generic
$(1,0)$ (respectively $(0,1)$) surfaces is empty. Moreover, for any $q\in\CP$
and $m\in\CPv$ such that $q\not\in m$, we have
$$
\qH\cap \Hm=L_{q,m}
$$
while, if $q\in m$, we get
$$\qH\cap \Hm=\PIqm.$$
It follows that the triple intersection
$$\Hm\cap H_{m'}\cap \qH = \pi_1^{-1}(m\times m')\cap\qH =
\{(m\times m',q\times(m\times m'))\}$$
is a single point. In conclusion,

\begin{proposition}\label{intersezione}
The intersection of Hirzebruch surfaces can be summarized by the products
\begin{align*}
\Of(1,0)\cd \Of(1,0)\cd\Of(1,0)=0,\\
\Of(1,0)\cd \Of(0,1)\cd\Of(1,0)=1,\\
\Of(0,1)\cd \Of(1,0)\cd\Of(0,1)=1,\\
\Of(0,1)\cd \Of(0,1)\cd\Of(0,1)=0.
\end{align*}
\end{proposition}

\begin{remark}\label{formulac2}
Thanks to the previous proposition, we are able to compute $c_{1}^{2}$ for a
generic surface $S$ of bidegree $(a,b)$. Indeed, given that
$\omega_{\FF}=\Oo(-2,-2)$, we also have that $\omega_{S}=\Oo(a-2,b-2)$, and
hence $c_{1}^{2}=\Oo(a-2,b-2)\cdot \Oo(a-2,b-2)\cdot \Oo(a,b)$. Therefore
\begin{align*}
c_{1}^{2}&=[(a-2)\Oo(1,0)+(b-2)\Oo(0,1)]\cdot[(a-2)\Oo(1,0)+(b-2)\Oo(0,1)]\cdot[a\Oo(1,0)+b\Oo(0,1)]\\
&=3a^{2}b+3ab^{2}-4a^{2}-4b^{2}-16ab+12a+12b.
\end{align*}
\end{remark}

We now show an important interplay between curves of bidegree $(1,1)$
and surfaces of bidegree $(0,1)$ and $(1,0)$.

\begin{lemma}\label{tutte11} 
  Let $C\subset \FF$ be a connected curve of bidegree $(1,1)$.
Then $C=\qH\cap \Hm$ for some $(q,m)\in\CPCPv$. In particular, if
  $C$ is smooth then $qm\ne0$ and $C=L_{q,m}$.
\end{lemma}

\begin{proof}
Since $\pi_1(C)$ is a line, it equals $m$ for some $m\in\CPv$, and
$C\subset\Hm$. Similarly, $C\subset\qH$ for some point
$q\in\CP$. Therefore
$$C\subseteq \qH\cap \Hm$$ and, since $\bdeg(C)=(1,1)$, we conclude
that $C\simeq \qH\cap \Hm$.
\end{proof}

\noindent Note that the union $\PUqm$ with $q\not\in m$ consists of two skew
lines, in contrast to \eqref{Lsing}. It is therefore a reducible $(1,1)$-curve,
given by two disjoint components of bidegree $(1,0)$ and $(0,1)$.\smallbreak

We now pass to work on the normal bundle. Recall \cite{OSS} that any
rank two vector bundle on $\PP^1$ is a direct sum of line bundles
$\Oo_{\PP^1}(a_1)\oplus\Oo_{\PP^1}(a_2)$.

\begin{lemma}\label{f1}
Let $C\subset \FF$ be any smooth rational curve. Then $C$ has bidegree $(1,1)$
if and only if its normal bundle $N_C$ is isomorphic to the direct sum of two
line bundles of degree $1$.
\end{lemma}

\begin{proof}
 Let $(d_1,d_2)$ be the bidegree of $C$ and let $N_{C}\simeq
 \Oo_{\PP^1}(a_1)\oplus\Oo_{\PP^1}(a_2)$.
  
Since $C$ is smooth, we have the exact sequence
\begin{equation}\label{eqf1}
0 \to T_C \to T_{\FF|C} \to N_C \to 0,
\end{equation}
where the first non-trivial map is the inclusion and the second is the
projection on the quotient.  Notice that since ${\FF}$ is homogeneous, then
$T_\FF$ is globally generated, and the same holds for $N_C$. Hence
$a_1,a_2\ge0$.

Since $\omega_\FF\cong \Of(-2,-2)$, we have $\det (T_\FF)\cong \Oo
_\FF(2,2)$. Thus $\deg (T_{\FF|C}) =2d_1+2d_2$. Moreover, since $C$ is
rational, then $T_C\simeq \Oo_{\PP^1}(2)$. Hence $\deg (T_C)=2$ and
$\deg(N_{C})=a_1+a_2$.  Therefore, from the exact sequence in \eqref{eqf1}, we
have
$$
a_1+a_2=2d_1+2d_2-2.
$$

Now if $a_1=a_2=1$ we have $d_1+d_2=2$ and we conclude by Remark \ref{rk2001}.

On the other hand, if $d_1=d_2=1$ we have $a_1+a_2=2$, with
$a_1,a_2\ge0$. Moreover $a_1\ne0\ne a_2$, since integral curves of
bidegree $(0,d)$ or $(d,0)$ have $d=1$, again by Remark \ref{rk2001}.
\end{proof}

This lemma is relevant to the deformation of twistor fibers, see
forward to Remark \ref{remarkintersectionj}.

\subsection{Surfaces of bidegree (0,d)} We add this subsection for
  completeness.  Let $C$ be a curve of degree $d$ in $\CP$ and consider the
  surface $S=\pi_1^{-1}(C)$.  The surface $S$ has bidegree $(0,d)$. Then the
  restriction
\[\pi_2|_S\colon S\lra \CPv\]
is a cover of degree $d$ of $\CP$, branched over the dual curve
$C^\vee$. Indeed, $\pi_2|_S^{-1}(\ell)=\{(p,\ell):p\in\ell\cap C\}.$ The following
lemma states that any $(0,d)$ surface in $\FF$ arises in this way:

\begin{lemma}\label{bideg0d} Given an integer $d>0$, if $S$ is an
  integral surface in $|\Of(0,d)|$, then $S=\pi_1^{-1}(C)$, where $C\subset\CP$
  is a degree $d$ {integral} curve.
\end{lemma}

\begin{proof}
Since $S\in|\Of(0,d)|$, we have $\dim\pi_1(S)\le1$, while $\pi_{2|S}$ is a
$d:1$ covering.  Moreover, it is easy to see that $\pi_1(S)$ is a plane
integral curve $C$ of degree $\tilde{d}$. Clearly, for a general line
$\Ll_m\subset\CP$, $C\cap\Ll_m=\{p_1,\ldots,p_{\tilde d}\}$.  Now $$S\cap
\Hm=\pi_1^{-1}(C\cap\Ll_m)= \bigcup_{k=1}^{\tilde d}\pi_1^{-1}(p_k).$$ Hence,
by Proposition \ref{intersezione}, for a general $q\in\CP$, we get $$\tilde
d=|(S\cap \Hm)\cap \qH|= \Oo_{F}(0,d)\cd \Oo_{F}(0,1)\cd\Oo_{F}(1,0)=d,$$ where
the last equality follows from Remark \ref{intersezione}.
\end{proof}

We omit the proof of the following result.

\begin{proposition}
Fix $d>0$ and let $S_1, S_2$ be integral surfaces of bidegree $(0,d)$. Let $C_i
= \pi_2(S_i)$ for $i=1,2$. Then $S_1$ and $S_2$ are isomorphic projective
varieties if and only if $C_1$ and $C_2$ are isomorphic projective varieties.
\end{proposition}

\section{Classification of surfaces of bidegree $(1,1)$}\label{Kronecker}

This section provides a description of all $(1,1)$-surfaces contained in the
flag manifold, and a classification of them up to projective equivalence. We
then proceed to classify the \emph{smooth} surfaces up to unitary equivalence,
in preparation for the twistor geometry that is introduced in the next
section. Recall that, by Remark \ref{p2}, the set of the $(1,1)$-surfaces has
complex dimension $7$.

Given a $3\times 3$ complex non-scalar matrix $A$, we define
\begin{equation}\label{SA}
  S_A=\{(p,\ell)\in \FF\mid pA\ell = 0\}.
\end{equation}  
This is a surface of bidegree $(1,1)$, since its intersection with a generic
fiber of each projection is one point. Indeed, being the zero locus of an
element of $\Ss_{1,1}$ (cf.\ \eqref{Ss}), it belongs to
$|\Oo(1,1)|$. Conversely, any surface of bidegree $(1,1)$ in $\FF$ will be
defined by an element of $\Pp_{1,1}$, equivalently by a suitable matrix $A$.
The surfaces of bidegree $(1,1)$ are therefore parametrized by the matrices of
the Kronecker pencil (see e.g.\ \cite[\S10.3]{Landsberg}) of the form $sA+tI$,
where $A$ is a complex $3\times 3$ non-scalar matrix and $s,t\in\CC$, with
$s\ne0$.

For each class, we can (i) choose a representative with a zero eigenvalue, and
(ii) if there is a non-zero eigenvalue we can assume it equals $1$. This can be
achieved with suitable choices $s$ and $t$. By listing the resulting Jordan
canonical forms, we deduce

\begin{lemma}\label{matricesA}
Any $(1,1)$-surface is projectively equivalent to $S_A$, where $A$ is one
of the following matrices:
\begin{align*}
&A_1=\left(\begin{array}{ccc}
  0&0&0\\
  0&1&0\\
  0&0&\la
\end{array}\right),\quad
A_2=\left(\begin{array}{ccc}
  0&0&0\\
  0&0&0\\
  0&0&1
  \end{array}\right),\quad
A_3=\left(\begin{array}{ccc}
  0&1&0\\
  0&0&0\\
  0&0&1
\end{array}\right),\\
\nonumber
& \hphantom{mmmmmm}
A_4=\left(\begin{array}{ccc}
  0&1&0\\
  0&0&0\\
  0&0&0
\end{array}\right),\quad
A_5=\left(\begin{array}{ccc}
  0&1&0\\
  0&0&1\\
  0&0&0
\end{array}\right),
\end{align*}
where  $\la\in\CC\sm\{0,1\}$.
\end{lemma}

We proceed to give geometric descriptions of some of the resulting classes of
surfaces.

\subsection{Smooth (1,1)-surfaces}
This class is a natural one to distinguish. Up to projective equivalence, it
can only arise from matrices of type $A_1$.

\begin{proposition}\label{DelP}
Let $S$ be a smooth $(1,1)$-surface in $\FF$. Then $S$ is a del Pezzo surface
of degree $6$, and is unique up to biholomorphism.
\end{proposition}

\begin{proof}
By the adjunction formula we have $\omega_S \cong\Of(-1,-1)_{|S}.$ Hence $S$ is
a del Pezzo surface of degree 6 (by Remark~\ref{formulac2}). From the
classification in \cite[\S8.4.2]{dolg} we get the uniqueness up to
biholomorphism.
\end{proof}

We can describe these del Pezzo surfaces more explicitly:

\begin{proposition}
If $A$ is a $3\times 3$ complex matrix $A$ that admits three distinct
eigenvalues, then the associated surface $S_A$ is smooth. It can be realized as
the blowup (i) of $\CP$ via $\pi_1$ at three points corresponding to the left
eigenvectors of $A$, or (ii) of $\CPv$ via $\pi_2$ at three points
corresponding to right eigenvectors of $A$.
\end{proposition}

\noindent By left (respectively, right) eigenvector, we mean a row
(respectively, column) vector satisfying the obvious equation. Of course, the
left eigenvectors are transposes of the (more usual) right eigenvectors of
$A^\vee$. But the left-right formalism is more in keeping with our approach.

\begin{proof}
For (i), let $p_1,p_2,p_3$ be points in $\CP$ corresponding to three linearly
independent left eigenvectors. We prove that the restriction of $\pi_1$ to
$S_A$ is a blowup of $\CP$ in such points. This is indeed a degree $6$ Del
Pezzo surface.

Given $q\in\CP$, we have
\[ \pi_1^{-1}(q)\cap S_A = \big\{(q,m)\in\CPCPv \mid qm=0,\ qAm=0\big\}.\]
The two conditions $qm=0$ and $qAm=0$ are dependent if and only if $q$ is a
left eigenvector of $A$. In this case, we have $\pi_1^{-1}(q)\subset
S_A$. Since the set $\pi_1^{-1}(q)\cap S_A$ is a point for generic $q$, we
conclude that $(\pi_1)_{|S_A}$ is the blowup of $\CP$ at $p_1,p_2,p_3$.

Case (ii) is similar.
\end{proof}

\begin{remark}\label{Xratio}
If the three left eigenvectors of $A$ correspond to points $p_1,p_2,p_3$ in
$\CP$, then the right eigenvectors of $A$ correspond to the lines $p_2\times
p_3,\>p_3\times p_1,\>p_1\times p_2$ in $\CPv$; see Figure~\ref{fig.A1}). A
triple of eigenvalues $(\alpha_1,\alpha_2,\alpha_3)$ is projectively equivalent
to $(0,1,\la)$, where $\la=(\alpha_3-\alpha_1)(\alpha_2-\alpha_1)$. Let
\begin{equation}\label{Lambda}
\Lambda = \Big\{\la,\ \frac1\la,\  1-\la,\  \frac1{1-\la},\ 
\frac\la{\la-1},\ \frac{\la-1}\la\Big\}.
\end{equation}
By permuting the eigenvalues, we see that the sets $\{0,1,\la'\}$ with
$\la'\in\Lambda$ are all projectively equivalent. We may regard $\la$ as a real
cross ratio.
\end{remark}

\begin{figure}
  \includegraphics[scale=1]{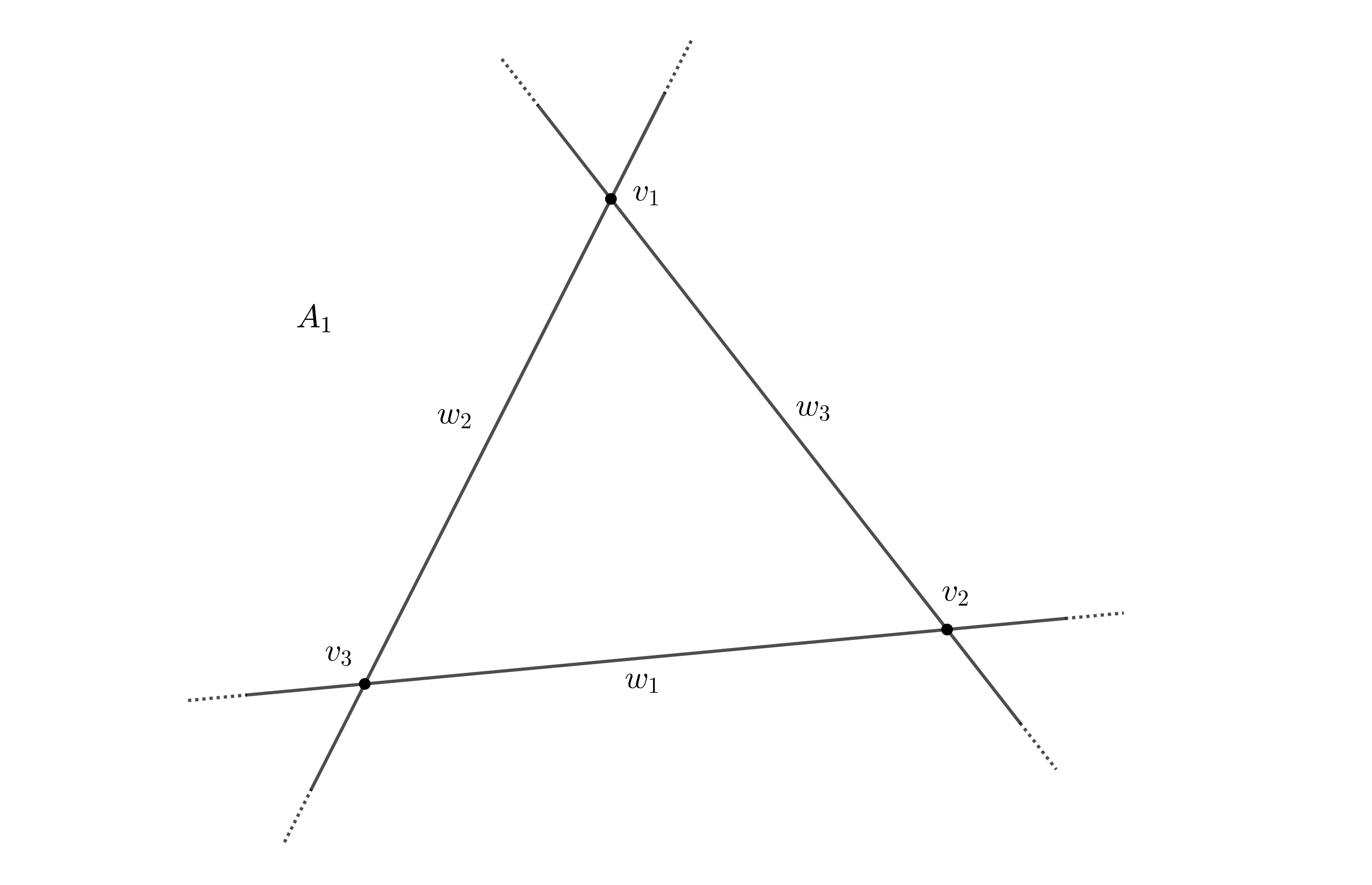}
  \vspace{-10pt}
\caption{A smooth $(1,1)$-surfaces of type $A_1$ can be seen as the blowup at
  three points in general position in either $\CP$ or $\CPv$.}
  \label{fig.A1}
\end{figure}

\subsection{Singular (1,1)-surfaces}
We now describe the set of singular surfaces of bidegree $(1,1)$. We shall
prove the next proposition directly using linear algebra, and then reconcile
its statements with Lemma \eqref{matricesA}.

\begin{proposition}\label{classification-singular}
The family of singular $(1,1)$-surfaces contained in $\FF$ is parametrized by
an irreducible variety of dimension $6$ and each irreducible singular
$(1,1)$-surface has exactly one singular point. The family of reducible
$(1,1)$-surfaces in $\FF$ has dimension $4$ and each reducible $(1,1)$-surface
is of the form $\qH\cup \Hm$, for some $(q,m)\in\CPCPv$.
\end{proposition}

\begin{proof}
Let $S_A$ be a $(1,1)$-surface defined by a matrix $A$. The Jacobian of the map
$\CC^3\times(\CC^3)^\vee\to\CC^2$ defined by
\begin{equation}
  \big((p_0,p_1,p_2),\>(\ell_0,\ell_1,\ell_2)\big)\longmapsto
  \left(\begin{array}{c}pA\ell\\p\ell\end{array}\right)
    \label{surface11}\end{equation}
can be represented by the matrix
\[ \left(\!\begin{array}{cc}A\ell & Ap\\\ell&p\end{array}\!\right).\]
It has rank less than $2$ at the point $(p,\ell)\in\FF$ if and only if $p$ is a
left eigenvector and $\ell$ is a right eigenvector of $A$, with the same
eigenvalue. Imposing these conditions, given that $p\ell=0$, we automatically
get $(p,\ell)\in S_A$.

Thus $S_A$ is singular if and only if the system
\begin{equation}\label{system-sing}
  \left\{\begin{array}{l}
pA=\la p\\
A\ell=\la\ell\\
p\ell=0
\end{array}\right.
\end{equation}
admits a non-trivial solution $(p,\ell)\in\CPCPv$. We want to prove that this
happens if and only if $\la$ has algebraic multiplicity greater than one. This
would imply that the family of singular $(1,1)$-surfaces has codimension $1$ in
the $7$-dimensional family of $(1,1)$-surfaces,  and (with reference to
Lemma \ref{matricesA}) is irreducible.

Suppose that $p$ is a left eigenvector of $A$ with eigenvalue $\la$ of
algebraic multiplicity $1$, and that $\mu$ is an eigenvalue of $A$ distinct
from $\la$. If $(A-\mu I)^2\ell=0$ then
\[ 0=p(A-\mu I)^2\ell= (\la-\mu)^2p\ell\]
so $p\ell=0$. It follows that the 2-dimensional annihilator of the row vector
$p$ is spanned by right eigenvectors or generalized eigenvectors with
eigenvalues distinct from $\la$. Hence, there is no non-zero column vector
$\ell$ that solves \eqref{system-sing}.

Suppose that $\la$ is an eigenvalue of $A$ of algebraic multiplicity
at least $2$. Denote by $V_\la$ and $W_\la$ the respective left and
right eigenspaces for $\la$. We have the following two cases.\\

\noindent(1) Suppose that $V_\la=\{p\}$ has dimension $1$, and set
$W_\la=\{\ell\}$.  Then there exists a generalized eigenvector $\tilde p$,
satisfying $\tilde p(A-\la I)=p$, and
\begin{equation}\label{genei}
  p\ell = \tilde p(A-\la I)\ell = 0.
\end{equation}  
It follows that $(p,\ell)$ is a solution of \eqref{system-sing}. The argument
above shows that there are no more solutions (projectively speaking).\\

\noindent(2) Suppose that $V_\la$ has dimension 2. In this case, $\PP(V_\la)$
represents an element $m\in\CPv$, and $\PP(W_\la)$ represents an element
$q\in\CP$. There are two subcases.
 \begin{enumerate}
 \item[(i)] Firstly, assume that there is an eigenvalue $\mu$ distinct from
   $\la$, and that $q$ is a left eigenvector for $\mu$. Since $A(p\times q)$ is
   a multiple of $(pA)\times(qA)$ by \eqref{cross}, it follows that $p\times
   q\in W_\la$, and $(p,p\times q)$ solves \eqref{system-sing} for any $p\in
   V_\la$. These solutions give rise to the curve
 \[ L_{q,m} = \{(p,\ell):p\in m,\ \ell\ni q\}\]
of singular points in $S_A$. It follows from \eqref{Hcross} that $S_A$ is the
reducible surface $\qH\cup\Hm$, where $qm\ne0$.

\item[(ii)] Secondly, assume that $A$ has a unique eigenvalue $\la$ (which we
  could take to be $0$). Then the solutions to \eqref{system-sing} are given by
  \[ \{(q,\ell): \ell\in W_\la\}\cup\{(p,m): p\in V_\la\} = 
  \pi_1^{-1}(q)\cup\pi_2^{-1}(m).\] We again have $S_A=\qH\cup \Hm$, but in
  this subcase $mq=0$.
 \end{enumerate}
Since we have seen that each reducible $(1,1)$-surface is of the form $\qH\cup
\Hm$ for $(q,m)\in \CP\times\CPv$, we conclude that the set of reducible
$(1,1)$-surfaces has complex dimension $4$.
  \end{proof}

\begin{figure}
\includegraphics[scale=.7]{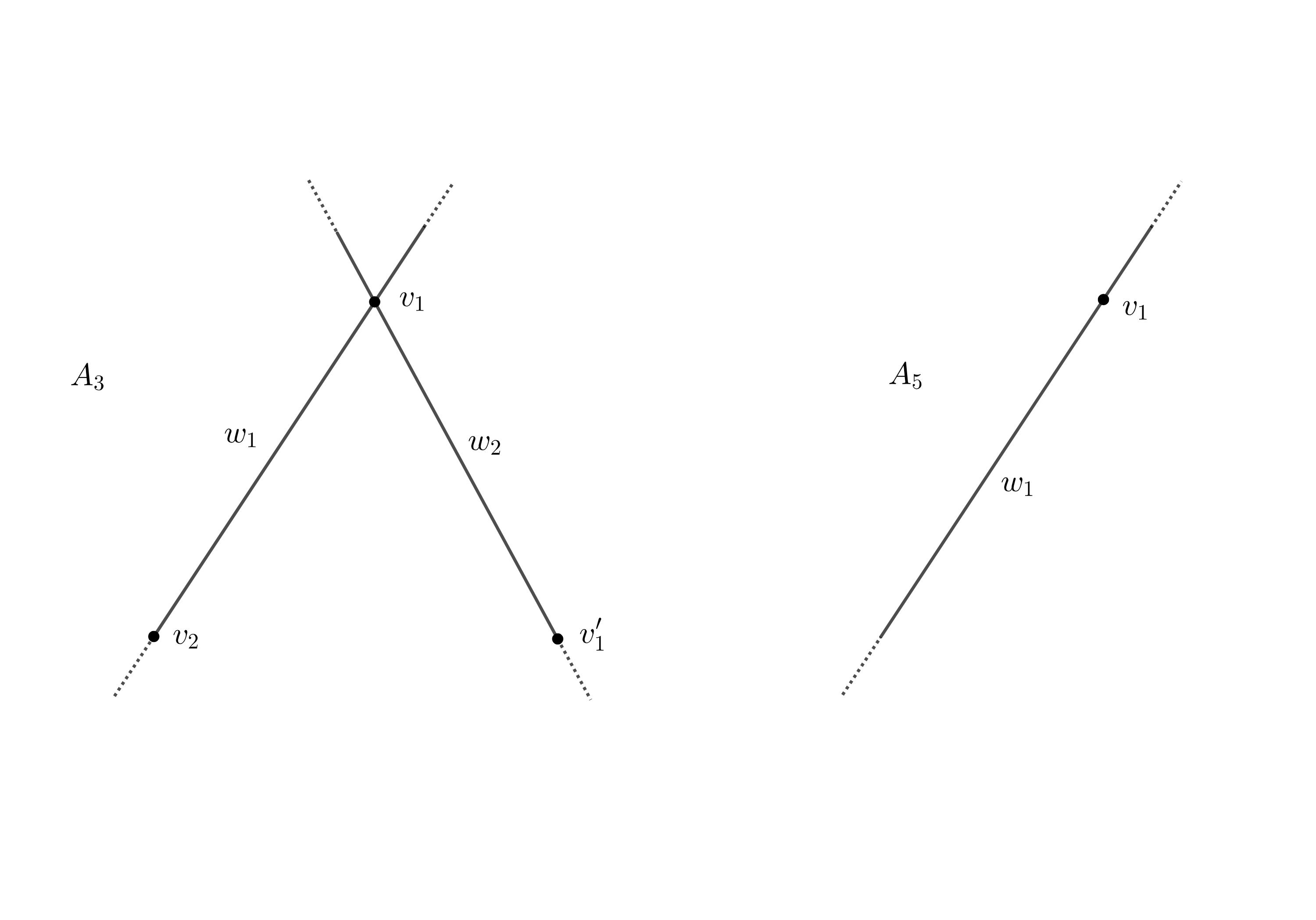}
\vspace{-50pt}
\caption{Non-smooth non-reducible $(1,1)$-surfaces $S_{A}$ have only one
  singular point. The disposition of the eigenvectors of $A$ and $A^\vee$ is
  described here.}
\end{figure}

\begin{figure}
  \includegraphics[scale=.7]{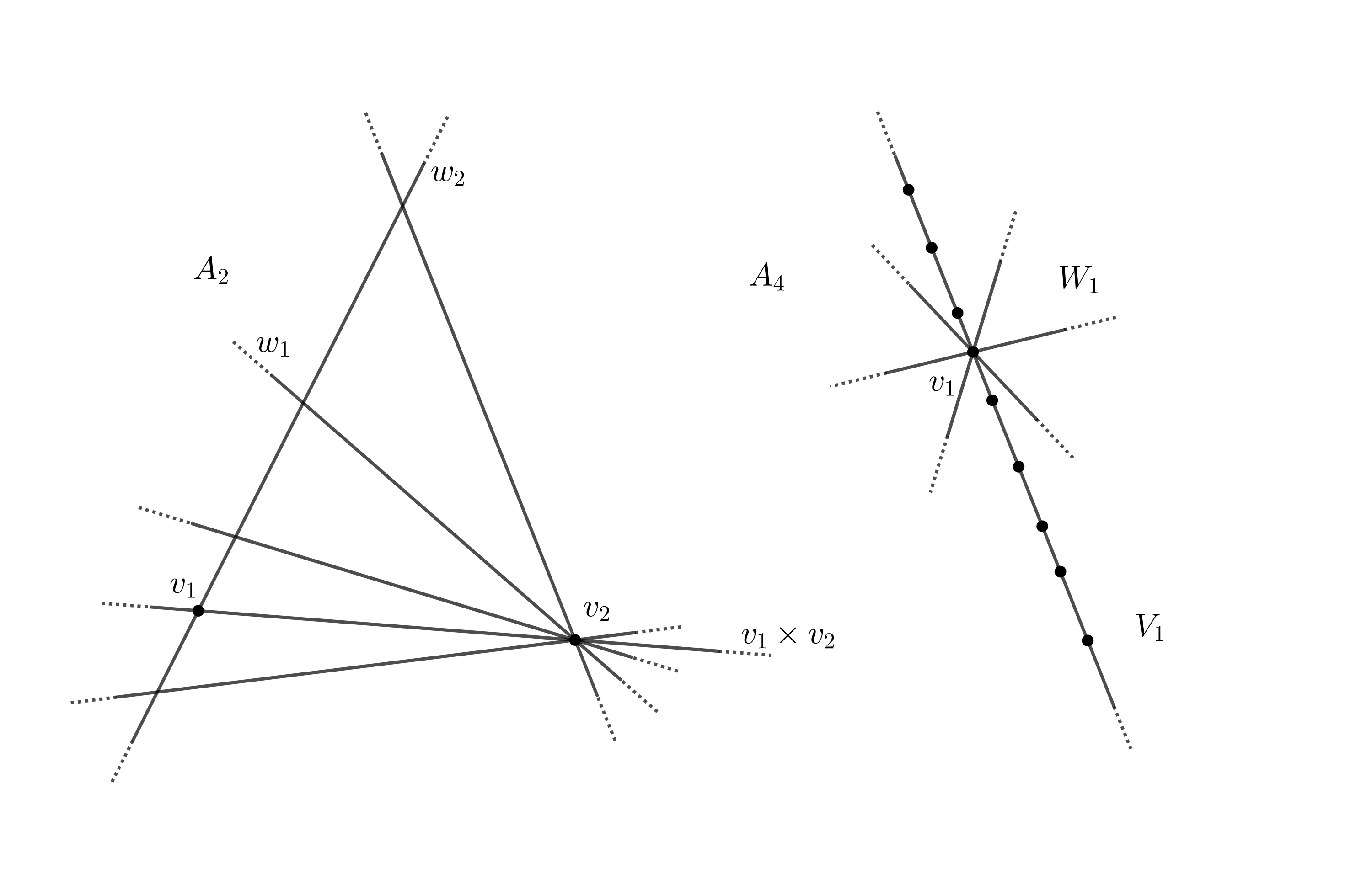}
  \vspace{-30pt}
\caption{Non-smooth reducible $(1,1)$-surfaces are singular along an
  element of $\bar\Vv$: either a curve $L_{q,m}$ for the case
  $A_2$ or the union of two intersecting fibers $\PUqm$ for $A_{4}$.}
\end{figure}

We end this section by summarizing the behaviour of the five canonical
cases: 
\begin{itemize}
\item $S_{A_1}$ is smooth.
\item $S_{A_2}={}_{[0:0:1]}H\cup H_{[0:0:1]}$ is reducible and singular on
\[ L_{[0:0:1],[0:0:1]}=\{([p_0:p_1:0],[\ell_0:\ell_1:0]^\vee)\mid
  p_0\ell_0+p_1\ell_1=0\}.\]
\item $S_{A_3}$ has only one singular point $([1:0:0],[0:1:0])$.
\item 
  $S_{A_4}={}_{[0:1:0]}H\cup H_{[1:0:0]}$ is reducible and singular on
\[\pi_1^{-1}[1:0:0]\cup \pi_2^{-1}[0:1:0]=\{([p_0:0:p_2],[0:\ell_1:\ell_2]\mid
  p_2\ell_2=0\}.\]
\item $S_{A_5}$ has a unique singular point $([1:0:0],[0:0:1])$.
\end{itemize}

\begin{remark}
Dolgachev in \cite{dolg} gives a classification of the singular del Pezzo surfaces of degree 6. In particular there exist two cases, both with one singular point: case (ii) of \cite[Table 8.4]{dolg} which contains 4 lines
and case (iii) containing 2 lines. It is easy to see that $S_{A_3}$ is of type (ii) and  $S_{A_5}$ is of type (iii). 
\end{remark}

\subsection{Unitary equivalence}\label{U3}
The realization of the flag manifold as the twistor space $\FF$ of $\CP$ will
require us to restrict to \textit{unitary automorphisms} of $\FF$, as defined
in Subsection \ref{Aut}.

Given a $(1,1)$-surface $S_A$ in $\FF$, we know that $A$ is projectively
equivalent to one of the matrices $A_i$ of Lemma \ref{matricesA}. Hence
there is a non-singular matrix $C$ such that
\[ C^{-1}\!AC=A_i,\]
By `$QR$-factorization', $C$ can itself be decomposed as $C=QR$ where $Q\in
U(3)$ and $R$ is upper triangular. Hence,
\[ Q^{-1}\!AQ = RA_iR^{-1}\]
is upper triangular. In other words, up to unitary equivalence, we can still
assume (at least) that $A$ is upper triangular.

Let us begin with the smooth case. Recall that $S_A$ is smooth if and only if
$A$ has three distinct eigenvalues. After adding a scalar multiple and
re-scaling, we can assume that these are $0,1,\la$ with $\la\in\CC\sm\{0,1\}$,
so that
\begin{equation}\label{forma-canonica}
A = \left(\begin{array}{ccc}
  0&a&b\\
  0&1&c\\
  0&0&\la
\end{array}\right).
\end{equation}

However, this representation is not unique. Let
\begin{equation*}
X = \left(\begin{array}{ccc}
  e^{i\vartheta_1+i\vartheta_2}&0&0\\
  0&e^{i\vartheta_2}&0\\
  0&0&1
\end{array}\right)
\end{equation*}
be a diagonal unitary matrix. Then
\begin{equation}\label{Torbit}
XAX^{-1} = \left(\begin{array}{ccc}
  0&ae^{i\vartheta_1}&be^{i\vartheta_1+i\vartheta_2}\\
  0&1&ce^{i\vartheta_2}\\
  0&0&\la
\end{array}\right)
\end{equation}
parametrizes the orbit of $A$ under the standard maximal torus $T^2$.

\begin{proposition}\label{classofA}
Let $A,A'$ be the matrices 
defined by respectively
\eqref{forma-canonica} and
 \[A' = \left(\begin{array}{ccc}
  0&a'&b'\\
  0&1&c'\\
  0&0&\la'
 \end{array}\right).\] 
Then $A$ and $A'$ are unitarily equivalent if and only if
$\la=\la'$ and $A'$ has the form \eqref{Torbit}.
\end{proposition}

\begin{proof}
Let us first assume that the matrices $A$ and $A'$ are unitarily equivalent, so
that there exists $X=(x_{ij})$ in $U(3)$ such that $XA=A'X$. This equation
translates into the system
\[\begin{cases}
0=x_{21}a'+x_{31}b'\\
0=x_{21}+x_{31}c'\\
0=x_{31}\la'\\
x_{11}a+x_{12}=x_{22}a'+x_{32}b'\\
x_{21}a+x_{22}=x_{22}+x_{32}c'\\
x_{31}a+x_{32}=x_{32}\la'\\
x_{11}b+x_{12}c+x_{13}\la=x_{23}a'+x_{33}b'\\
x_{21}b+x_{22}c+x_{23}\la=x_{23}+x_{33}c'\\
x_{31}b+x_{32}c+x_{33}\la=x_{33}\la',
\end{cases}\]
and immediately gives $x_{31}=x_{21}=x_{32}=0$. But as $X$ is unitary and now
upper triangular, it must be diagonal. Thus, $x_{13}=x_{12}=x_{23}=0$ and
$|x_{11}|=|x_{22}|=|x_{22}|=1$. Substituting above, we get
\[\begin{cases}
x_{11}a=x_{22}a'\\
x_{11}b=x_{33}b'\\
x_{22}c=x_{33}c'\\
x_{33}\la=x_{33}\la',
\end{cases}\]
from which we get
\begin{equation}{abc}
  \la=\la';\quad |a'|=|a|,\ |b'|=|b|,\ |c'|=|c|;\quad ab'c=a'bc'.
\end{equation}
This implies that $a'=e^{i\vartheta_1}a, c'=e^{i\vartheta_2}c$ and
$b'=e^{i(\vartheta_1+\vartheta_2)}b$, and leads to case (i). Conversely, we have already seen that any
matrix \eqref{Torbit} is unitarily equivalent to $A$.
\end{proof}

\begin{remark}
Recall that the surface   of bidegree $(1,1)$ $S_A$ defined by a matrix $A$ of the form
\eqref{forma-canonica} is invariant if we modify $A$ by adding multiples of the identity and rescaling. 
But in
order to return to the canonical form we would need to apply an element of
the Weyl group of $\SU(3)$, which acts by permuting the standard coordinates of
$\CC^3$. Or, we could reverse the order and start by transforming $A$ into one
of
\[
\left(\begin{array}{ccc}
  0&b&a\\
  0&\la&0\\
  0&c&1\\
\end{array}\right),\quad
\left(\begin{array}{ccc}
  \la&0&0\\
  c&1&0\\
  b&a&0\\
\end{array}\right),\quad
\left(\begin{array}{ccc}
  1&0&c\\
  a&0&b\\
  0&0&\la\\
\end{array}\right),\quad
\left(\begin{array}{ccc}
  \la&0&0\\
  b&0&a\\
  c&0&1\\
\end{array}\right),\quad
\left(\begin{array}{ccc}
  1&c&0\\
  0&\la&0\\
  a&b&0\\
\end{array}\right).\]
The first three matrices are the result of applying a 2-cycle, and the last two
of applying a 3-cycle. Then in special cases, if one or more of $a,b,c$ vanish in $A$, the new matrices can be put in to canonical forms as follows:
\[
\left(\begin{array}{ccc}
  0&\frac b\la&\frac a\la\\
  0&1&0\\
  0&0&\frac1\la\\
\end{array}\right),\
\left(\begin{array}{ccc}
  0&0&0\\
  0&1&0\\
  0&0&\frac\la{\la-1}\\
\end{array}\right),\
\left(\begin{array}{ccc}
  0&0&-c\\
  0&1&-b\\
  0&0&1\!-\!\la\\
\end{array}\right),\
\left(\begin{array}{ccc}
  0&0&0\\
  0&1&-\frac a\la\\
  0&0&\frac{\la-1}\la\\
\end{array}\right),\
\left(\begin{array}{ccc}
  0&\frac c{\la-1}&0\\
  0&1&0\\
  0&0&\frac1{1-\la}\\
\end{array}\right).\]

Then if we denote by $S{(\la,a,b,c)}$ the surface defined by a matrix $A$ of the form \eqref{forma-canonica}, we deduce that
$S{(\la,0,0,0)}$
is unitarily equivalent to $S{(\la',0,0,0)}$ for any $\la'\in\Lambda$ defined in Remark \ref{Xratio}.
Moreover, we have the following relations of unitarily equivalence:
$$S{(\lambda,a,b,0)}\sim S{\left(\frac1\la,\frac{b}\la,\frac{a}\la,0\right)}\qquad
S{ (\lambda,0,b,c)}\sim S{(1-\la,0,-c,-b)},$$
$$S{(\lambda,a,0,0)}\sim S{\left(\frac{\la-1}\la,0,0,-\frac{a}\la\right)}, 
\quad \quad S{(\lambda,0,0,c)}\sim S{ \left(\frac1{1-\la},\frac{c}{1-\la},0,0\right)}.$$

Notice that, even in the general case, when  $abc\neq0$, it is possible to find different canonical forms. For example consider the following example:
Let $A$ of the form \eqref{forma-canonica}, with $a\in\RR$
 and take
$$
X=\left(\begin{array}{ccc}
 -\frac{a}{\sqrt{a^2+1}}&-\frac{1}{\sqrt{a^2+1}}&0\\
 \frac{1}{\sqrt{a^2+1}}&-\frac{a}{\sqrt{a^2+1}}&0\\
  0&0&-1
\end{array}\right).
$$
It is easy to check that $X$ is unitary, and to compute
that 
$$
A'=XAX^T=\left(\begin{array}{ccc}
  1&a&\frac{(ba+c)}{\sqrt{a^2+1}}\\
  0&0&\frac{(ac-b)}{\sqrt{a^2+1}}\\
  0&0&\la
\end{array}\right),
$$
hence, for $a\in\RR$, we obtain that 
\[ S{(\lambda,a,b,c)}\sim
S{\left(1-\la,-a,-\frac{(ba+c)}{\sqrt{a^2+1}},-\frac{(ac-b)}{\sqrt{a^2+1}}\right)}.\]
\end{remark}

\begin{corollary}
Let $S_A$ be a smooth $(1,1)$-surface where $A$ is as in
Formula~\eqref{forma-canonica}, with $[a:b:c]\in\CP$. Then its
stabilizer is trivial.
\end{corollary}

\begin{proof}
  The proof is a direct consequence of the computations given to prove
  Proposition~\ref{classofA} applied to $A$ and $A'=A$.
\end{proof}

\section{Twistor theory}\label{Twy}

We now study the flag manifold as the twistor space of $\CP$, in the context of
4-dimensional Riemannian geometry. Following~\cite[Example 3, p.~438]{ahs},
\cite[\S4.3, p.~500]{gauduchon}, \cite[Remark p.147]{hitchin} and \cite[Ex
  p.~0]{lebrun2}, we construct the twistor projection explicitly.

Consider the Hermitian product
\[\langle p,q\rangle = pq^* = p_0\bar q_0+p_1\bar q_1+p_2\bar q_2,\]
where $p=(p_0,p_1,p_2)$ and $q=(q_0,q_1,q_2)$ are row vectors, and
$q^*=\bar q^\vee$. This pairing induces the anti-linear bijection
\begin{equation}\label{bij}
  \CC^3\to(\CC^3)^\vee,\quad q\mapsto q^*.
\end{equation}
We shall extend the asterisk notation to projective classes by writing
\[ p=[p_0:p_1:p_2]\in\CP \quad\Rightarrow\quad
p^*=[\bar p_0:\bar p_1:\bar p_2]^\vee\in \CPv,\]
and $p^*$ can be thought of as the line at infinity $\ell_p^\8$. We can
similarly convert a line into a point:
\[ \ell=[\ell_0:\ell_1:\ell_2]^\vee\in\CPv \quad\Rightarrow\quad
\ell^*=[\bar\ell_0:\bar\ell_1:\bar\ell_2]\in\CP.\]

The twistor projection $\pi:\FF\to\CP$ is defined by
\begin{equation}\label{pi}
  \pi(p,\ell) = p^*\times\ell.
  \end{equation}
Recall that this cross product is the point of intersection of the two
lines, i.e.\ $p^*\cap\ell=\{p^*\times\ell\}$. The equation
\[\pi([p_0:p_1:p_2],[\ell_0:\ell_1:\ell_2]^\vee) = 
[\bar{p_1}\ell_2-\bar{p_2}\ell_1:\bar{p_2}\ell_0-\bar{p_0}\ell_2:
       \bar{p_0}\ell_1-\bar{p_1}\ell_0]\]
     gives a coordinate representation of $\pi$.

The introduction of the bijection \eqref{bij} allows us to regard the
projections $\pi_1,\pi_2,\pi$ on the same footing. If we set
$\pi_2^*(p,\ell)=\ell^*$, then $\pi_1,\pi,\pi_2^*$ are all maps $\FF\to\CP$.
A point $(p,\ell)\in\FF$ defines a unitary splitting
\begin{equation}\label{osum}
  \CC^3 =
\langle p\rangle\oplus\langle p^*\times\ell\rangle\oplus\langle\ell^*\rangle,
\end{equation}
and our projections now correspond to the three components. Note that
$\pi_2^*$ is now anti-holomorphic, while the twistor projection $\pi$
is neither holomorphic nor anti-holomorphic.

\[
\begindc{\commdiag}[3]
\obj(0,0)[A]{$\CP$}
\obj(260,0)[Aa]{$\CP$}
\obj(130,130)[B]{$\FF$}
\obj(130,260)[C]{$\CP$}
\mor{B}{A}{$\pi_1$}[\atright,\solidarrow]
\mor{B}{Aa}{$\pi_2^*$}[\atleft,\solidarrow]
\mor{B}{C}{$\pi$}[\atleft,\solidarrow]
\enddc\]\vs

\subsection{Almost complex structures}\label{acs}
Fix a point $x=(p,\ell)$ in $\FF$, and set $q=\pi(x)=p^*\times\ell$. Referring
to \eqref{osum}, we have an identification
\begin{equation}\label{TxFF}
(T_x\FF,\JJ) \cong
  (\langle p^*\rangle\otimes\langle\ell^*\rangle)\ \oplus\
  (\langle p^*\rangle\otimes\langle q\rangle)\ \oplus\
  (\langle q^*\rangle\otimes\langle\ell^*\rangle).
\end{equation}
of the holomorphic tangent space to $\FF$ for the complex structure $\JJ$ that
arises from $\CPCPv$, which we have been considering. These choices ensure that
the projections $\pi_1,\pi,\pi_2$ behave appropriately with respect to $\JJ$,
in particular that $\pi_1$ is holomorphic. It is well known that
$\pi_1\colon\FF\to\CP$ can be identified with the projectivized tangent bundle
$\PP(T\CP)$. The same is true of the non-holomorphic fibration
$\pi\colon\FF\to\CP$ except that in this case the realization of $\JJ$ is more
complicated; this is the heart of Penrose's twistor space concept
\cite{ahs,OR}.

The analogue of \eqref{TxFF} for the standard complex structure $J$ on the
projective plane is
\begin{equation}\label{TxCP}
(T_q\CP,J) \cong
  (\langle q^*\rangle\otimes\langle\ell^*\rangle)\ \oplus\
  (\langle q^*\rangle\otimes\langle p\rangle)
  = L\oplus L^\top.
\end{equation}
Geometrically, $L$ is tangent to the line $\ell$ and $L^\top$ is tangent to the
line $p^*$ (both lines pass through $q$). We can now see that \eqref{TxFF} is
`constructed' by combining a vertical component (the first summand) with a
horizontal component whose almost complex structure $J'$ is a twisted version
of \eqref{TxCP} defined by setting
\[ J' = \left\{\begin{array}{ll}
J & \hbox{ on }L\\
-J & \hbox{ on }L^\top,\end{array}\right.\]
In this way, each fiber $\pi^{-1}(q)\simeq U(2)/T^2$ parametrizes almost
complex structures on $T_q\CP$ that are orthogonal relative to the
Fubini-Study metric:
\begin{equation}\label{gFS}
  \gFS(J'X,J'Y)=\gFS(X,Y).
\end{equation}
In the context of this paper, we record

\begin{definition}\label{OCS}
An \emph{orthogonal complex structure} is a
complex structure defined on an open subset of $\CP$ satisfying \eqref{gFS} at
each point.
\end{definition}

\subsection{Symmetries}
One can relate the fibers of $\pi_1,\pi,\pi_2^*$ by exploiting the obvious
3-fold symmetry inherent in the unitary description of $\FF$.

\begin{definition}\label{jjj} Let $j_1,j,j_2$ be the diffeomorphims of
    $\FF$ defined by
\[\begin{array}{rcl}
j_1(p,\ell) &=& (p,p\times\ell^*)\\
j(p,\ell) &=& (\ell^*,p^*)\\
j_2(p,\ell) &=& (p^*\times\ell,\ell).
\end{array}\]
\end{definition}

Each of these transformations are involutions. For example,
\[  (j_1)^2(p,\ell)=(p,p\times(p\times\ell^*)^*)=(p,p\times(p^*\times\ell)),\]
but the projective class $p\times(p^*\times\ell)$ equals $\ell$ by the
well-known vector identity, given that the scalar product $p\ell$
vanishes. Each involution $j_1,j,j_2$ permutes the projections
$\pi_1,\pi,\pi_2^*$. In particular,
\begin{equation}\label{12j}
  \pi = \pi_1\circ j_2 = \pi_2^*\circ j_1.
\end{equation}  
The involutions therefore generate the symmetry group $S_3$ that permutes the
projections $\pi_1,\pi,\pi_2^*$. Notice that (only) $j$ is anti-holomorphic.

\begin{remark}\label{JJ}
The group $S_3$ we have just defined is a representation of the Weyl group $W$
of $SU(3)$, mentioned in Subsection \ref{Aut}. To understand this, regard a
point of $\FF$ as a right coset $g\TT$, where $\TT=T^2$ is the isotropy
subgroup of $SU(3)$ fixing a point (the `origin') of $\FF$. If $w\in N(\TT)$
then we can define an action
\[ w\cdot(g\TT) = (g\TT)w^{-1} = (gw)\TT.\]
Suppose that $w\not\in \TT$ (so that $w$ is not the identity in $W$). Then
$w$ does not act on $\FF$ as a holomorphic isometry (that would be
the action $g\TT\mapsto wg\TT$), and fixes no point $g\TT\in\FF$.

The six complex structures induced on $\FF$ by $S_3$ are precisely the $\SU(3)$
invariant complex structures considered in \cite{BS}. Each can be defined in
terms of decomposition analogous to \eqref{TxFF}, which was used to specify the
standard one $\JJ$. These six structures are supplemented by two non-integrable
almost complex structures $\pm J'$, which render the three fibrations
$\pi_1,\pi,\pi_2$ equivalent. Moreover, even permutations in $S_3$ act
$J'$-holomorphically on $\FF$, and it is this fact that allows one to generate
harmonic maps of Riemann surfaces into $\CP$ from holomorphic ones \cite{es}.
\end{remark}  

Any element of the group $S_3$ generated by Definition \ref{jjj} commutes with
the action of the group of unitary automorphisms discussed in Subsection
\ref{Aut}. For example,
\[B\cdot j(p,\ell) = (\ell^*B,Bp^*) = ((B\ell)^*, (pB^{-1})^*)
  = j(B\cdot (p,\ell)),
  \]
for $B\in\SU(3)$. Moreover, these equations yield

\begin{lemma}
A projective automorphism of $\FF$ induced by $B\in\SL(3,\CC)$ is unitary if
and only if $B\circ j = j\circ B$.
\end{lemma}

A unitary automorphism of $\FF$ commutes with each of the three projections
$\pi_1,\pi,\pi_2$, and is completly determined by its action on the base $\CP$
of any one of these projections. The relevance of an anti-holomorphic
involution in the definition of twistor spaces is discussed in
\cite{ahs,gauduchon,hitchin,lebrun,lebrun2}.

\begin{remark}\label{secant}
Reducing the structure group $\SU(3)$ further to the orthogonal group
$SO(3)$, one can define complex conjugation as an anti-holomorphic
involution $(p,\ell)\mapsto(\bar p,\bar\ell)$. Composing this involution
with $j$ gives rise to the holomorphic involution \eqref{kappa} of
$\FF$, and the latter covers complex conjugation $\si$ in $\CP$
relative to the twistor fibration $\pi$ since
\[\pi(\ell^\vee,p^\vee)=(\ell^\vee)^*\times p^\vee=\bar{\ell\times p^*}.\]
It is well known that the orbifold $\CP/\!\left<\si\right>$ is
homeomorphic to $S^4$, thought of as the sphere inside the real
irreducible 5-dimensional representation of $SO(3)$. Mapping
$(p,\ell)$ to the symmetric traceless $3\times3$ matrix $\ell p+(\ell
p)^\vee$ realizes the quotient $\FF/\!\left<\kappa\right>$ as the
secant variety of a rational normal curve in $\PP^4$ (defined by
setting $p^\vee=\ell$). This quotient corresponds to the twistor space
of $S^4$ endowed with one of a series of $SO(3)$-invariant
orbifold self-dual Einstein metrics \cite{hitchin4}.
\end{remark}

\subsection{Twistor fibers}
This subsection analyses more carefully the fibers of $\pi$.

\begin{definition}
Given $q\in\CP$ we call $\pi^{-1}(q)$ a \textit{twistor fiber} or
\textit{twistor line}.
\end{definition}

Given $q\in\CP$, we have
\[\pi^{-1}(q) = \{(p,\ell)\in\FF \mid p^*\times\ell=q\} =
\{(p,\ell)\in \FF \mid q\ell=0,\ pq^*=0\}.\]
From the general theory in the citations above, we know that the fibers of
$\pi$ are rational $j$-invariant curves with normal bundle
$\Oo(1)\oplus\Oo(1)$. Lemma~\ref{tutte11} and~\ref{f1} tell us each curve
$L_{q,m}$ has the same properties once we impose $j$-invariance. Indeed,
Definition \ref{def rette} implies that
\[ \pi^{-1}(q) = L_{q,q^*}.\]
In particular, any twistor line also has bidegree $(1,1)$.

The family of twistor fibers is obtained as the set of fixed points of the
anti-holomorphic involution $j$ acting on the space of parameters of the closure
of $(1,1)$-curves. Whence,

\begin{lemma}
The set of twistor fibers is a Zariski dense subset of $\Vv$.
\end{lemma}

Notice the analogy between the previous lemma and~\cite[Lemma 3.2]{altavillaballico1}.

\begin{remark}\label{remarkintersectionj}
Since $j(L_{q,m})=L_{m^*,q^*}$, the curve $L_{q,m}$ is $j$-invariant
if and only if it is a twistor fiber. The latter define a real
(i.e.\ $j$-invariant) slice of $\Vv$, see Definition~\ref{def
  rette}. Conversely, $\bar\Vv\cong\CP\times\CP$ can be thought
of as the compactified complexified space of twistor fibers, analogous
to the role that the Klein quadric $\Gr_2(\CC^4)$ plays in Penrose's
twistor theory for $\PP^3$.
\end{remark}

Given the vector identity
\[ (m\times q^*)(q\times m^*) = -|m\times q^*|^2\]
(before we projectivize), Lemma~\ref{lemmaintersection11} implies that
$L_{q,m}\cap j(L_{q,m}) = L_{q,m}\cap L_{m^*,q^*}$ is empty if and only if
$q,m^*$ are distinct points of $\CP$, i.e.\ $L_{q,m}$ is not a twistor fiber.
On the other hand, provided $pq^*\ne0$,
\[ \pi_1^{-1}(p)\cap j(\pi_1^{-1}(q)) =
\pi_1^{-1}(p)\cap\pi_2^{-1}(q^*) = \es,\]
so in this case $j$ generates a skew pair of fibers.

Thanks to Remark~\ref{intersectwithfbire}, we are able to compute the twistor
projections $\pi(\pi_i^{-1}(q))$ of curves of bidegree $(0,1)$ ($i=1$) and
$(1,0)$ ($i=2$). For any $q\in\CP$, we have
$$
\pi(\pi_1^{-1}(q))=\{z\in\CP\mid \pi^{-1}(z)\cap\pi_1^{-1}(q)\ne\es\},
$$ but $\pi^{-1}(z)=L_{z,z^*}$, hence
$\pi^{-1}(z)\cap\pi_1^{-1}(q)\neq\es$ if and only if $qz^*=0$.
Collecting everything:
$$
\pi(\pi_1^{-1}(q))=\{z\in\CP\mid zq^*=0\},
$$
Analogously, 
$$
\pi(\pi_2^{-1}(m))=\{z\in\CP\mid zm=0\}.
$$
Each image is a projective line. The same conclusion can be reached by
means of \eqref{12j}.

\section{Curves and surfaces revisited}\label{Twn}

We begin this section by revisiting the description of surfaces of bidegree
$(1,0)$ and $(0,1)$. Any two surfaces of type $(0,1)$ (respectively, $(1,0)$)
are equivalent under unitary transformation. Using the projections
$\pi_1,\pi_2$, this is a consequence of the fact $\SU(3)$ acts transitively on
both $\CP$ and $\CPv$.

\begin{lemma}
  $\qH$ is invariant by $j_2$, and $\Hm$ is invariant by $j_1$.
\end{lemma}

\begin{proof}
  This follows immediately from Definitions \ref{qHm} and \ref{jjj}. The
  condition that a point of $\FF$ lies in $\qH$ depends only on its image under
  $\pi_2^*$. But this is unchanged by $j_2$. Similarly for the triple $\Hm,
  \pi_1,j_1$.
\end{proof}

Recall Corollary \ref{fritz}: The restriction of $\pi_1$ realizes
$\qH$ as the blow-up of $\CP$ at $q$. The same is true of $\pi$, since
\[ \pi^{-1}(p)= (\pi_1\circ j_2)^{-1}(p) = j_2(\pi_1^{-1}(p)),\]
which is a single point unless $p=q$. A similar argument applies to $\Hm$:

\begin{corollary}\label{piH}
  The twistor projection $\pi$ realizes $\qH$ as the blow-up of $\CP$ at $q$,
  and $\Hm$ as the blow up of $\CP$ at $m^*$.
\end{corollary}

\subsection{Orthogonal complex structures}
We are now in a position to use the theory of Subsection \ref{acs} to deduce
an application of Corollary \ref{piH}. Recall the concept (Definition
\ref{OCS}) of an orthogonal complex structure (that we shall abbreviate OCS)
relative to the Fubini-Study metric.

\begin{proposition}
The complement of a point $q$ in $\CP$ admits an OCS $J_q$, inducing the
opposite orientation to the standard complex structure $J$. Moreover, the
action of $J_q$ on $T_p\CP$ is defined by reversing the sign of $J$ on the
complex line $L$ tangent to the projective line containing $p$ and $q$.
\end{proposition}

\begin{proof}
The complex structure $J_q$ is that of $\qH\sm\pi^{-1}(q)$, which $\pi$ bijects
onto $\CP\sm\{p\}$. The fact that this complex structure satisfies \eqref{gFS}
follows from the remarks above. The second statement is a consequence of the
fact that the complex surface $\qH$ is a union of fibers $\pi_2^{-1}(\ell)$
each of which projects to a line $\ell$ passing through $q$.
\end{proof}

\begin{remark}
  Note that $J_q$ does not extend to $\CP$, for one thing $\pi$ has
no global sections over $\CP$ even topologically. Thus, $\CP\sm\{q\}$ is a
maximal domain of definition for $J_q$.
\end{remark}

The results of this subsection point to the analogy between the family of
$(1,0)$ (respectively, $(0,1)$) surfaces in the twistor space $\FF$ and
projective planes in the twistor space $\PP^3$. Two members of each family
intersect in a projective line, and each member contains exactly one twistor
fiber, cf.\ \cite{sv1,shapiro}.

In the next subsection, we shall effectively illustrate the appearance of some
genus zero $J_q$-holomorphic curves in the Fubini-Study domain $\CP\sm\{q\}$.

\subsection{Spheres and dumbbells}             
We shall now consider the effect of the twistor fibration on the simplest
holomorphic objects of bidegree $(1,1)$ that we have defined, namely the
curves
\[ L_{q,m} = \{(p,\ell)\mid p\ell=0,\ q\ell = 0,\ pm=0\}.\]
They form a family $\bar\Vv$ that includes two special cases:
\begin{itemize}
  \item If $q\in m$ then $\pi$ maps $L_{q,m}=\PUqm$ to the union $q^*\cup m$ of
    the two lines in $\CP$ intersecting in the point $q^*\times m$. For
    example, if $q=[1:0:0]$ and $m=[0:1:0]^\vee$ then
\[ q^*\cup m=\{[z_0:z_1:z_2]:z_0=0\hbox{ or }z_1=0\}.\]
\item If $q^*=m$ so that then our curve $L_{q,m}=\pi^{-1}(q)$ is a twistor
  fiber, and projects to the single point $q\in\CP$.
\end{itemize}
In the generic case, $L_{q,m}$ is biholomorphic to $\PP^1$, and $\pi$ embeds it
in $\CP$. For if $(p,\ell)\in L_{q,m}$ and $\pi(p,\ell)=r$ then $\ell=r\times
q$ and $p=\ell\times m$. The image of $L_{q,m}$ will therefore be homeomorphic
to a 2-sphere. Our next results make this precise.

\begin{lemma}
Let $L_{q,m}\in\Vv$, with $q=[q_0:q_1:q_2]$, $m=[m_0:m_1:m_2]$ and
$qm\ne0$. Then $\pi(L_{q,m})=$ $\{z\in\CP \mid z\,\Phi z^* = 0\}$, where
\begin{equation}\label{Phi}
\Phi = \left(\begin{matrix}
m_1q_1+m_2q_2 & -m_0q_1 & -m_0q_2\\
-m_1q_0 & m_0q_0+m_2q_2 & -m_1q_2\\
-m_2q_0 & -m_2q_1 & m_0q_0+m_1q_1
\end{matrix}
\right).
\end{equation}
\end{lemma}

\begin{proof}
The set $\pi(L_{q,m})$ coincides with the set of points $z\in\CP$ such that the
fiber $\pi^{-1}(z)$ meets $L_{q,m}$. But
$\pi^{-1}(z) = L_{z,z^*}$, so
\[ \pi(L_{q,m})=\{z\in\CP \mid L_{z,z^*}\cap L_{q,m}\ne\es\}.\]
In view of Lemma~\ref{lemmaintersection11} and
Corollary~\ref{condition-intersection}, we deduce that $L_{q,m}$ and
$L_{z,z^*}$ have a non-empty intersection if and only if (choosing vector
representatives) the expression
\[ (z^*\times m)(z\times q) = (qm)(zz^*) - (zm)(qz^*)\]
vanishes. The right-hand side equals $z\Phi z^*$ where $\Phi$ is the
$3\times3$ matrix $(qm)I-mq$, where $qm=q_0m_0+q_1m_1+q_2m_2$.
whose entries are as stated.
\end{proof}

\begin{proposition}
Let $qm\ne0$ and $q^*\ne m$. Then $\pi$ maps $L_{q,m}$ onto a round 2-sphere in
real affine coordinates.
\end{proposition}

\begin{proof}
The matrix $\Phi=(qm)I-mq$ is singular, we see immediately that $q\Phi=0$ and
$\Phi m=0$. In order to analyse the equation $z\Phi z^*=0$ we can first apply a
unitary transformation so as to assume that
\[ q = [1:0:0],\qquad m=[1:2\rho:0],\]
with $\rho$ a positive real number (which will turn out to be the radius of our
sphere).
\[ \Phi=\left(\begin{matrix}
0 & 0 & 0\\
-2\rho & 1 & 0\\
0 & 0 & 1
\end{matrix}
\right),\]
and we obtain
\begin{equation}\label{mm}
  |z_1|^2+|z_2|^2 - 2\rho\bar z_0z_1 = 0.
\end{equation}

There are no solutions with $z_0=0$, so we set $z_0=1$ and work with the affine
coordinates
\[ (z_1,z_2) = (x_1+iy_1,\>x_2+iy_2),\]
and their real and imaginary components on the right-hand side. Then
\[\left\{\begin{array}{l}
x_1^2+y_1^2+x_2^2 + y_2^2 - 2\rho x_1 = 0\\
y_1=0.
\end{array}\right.\]
These equations obviously define a 2-sphere in the Euclidean space $\RR^3$
defined by $y_1=0$. For the remark below, it will be convenient to set
$x=x_1-\rho$, so that the 2-sphere becomes
\[  x^2+x_2^2+y_2^2 = \rho^2\]
and with centre $(x,x_2,y_2)=(0,0,0)$.
\end{proof}

Expanding slightly the proof verifies the special cases discussed above. In the
limit $m_1\to0$ (so $q^*=m$) the 2-sphere degenerates to a single point. If, on
the other hand, we had allowed $m_0=0$, we would have been led to the union of
the projective lines $z_0=0$ and $z_1=0$, again as expected. These observations
are illustrated by Figure \ref{rot}, whose justification we work towards.

The 2-sphere $\pi(L_{q,m})$ is not homogeneous for the Fubini-Study
metric because it is not an orbit of $SU(3)$. However, it is a union
of circles
\[ \{(x,x_2,y_2):x^2+x_2^2+y_2^2 = \rho^2\}\]
that are orbits of $U(1)\subset \SU(3)$. It is therefore isometric to a surface
of revolution with the induced Fubini-Study metric. In order to identify the
shape of this, we shall use the methods of \cite[\S8]{ABS}. Set
\[ z_1=x=\rho\sin v,\qquad z_2=x_2+iy_2=\rho e^{iu}\!\cos v,\]
so the $u$ and $v$ represent longitude and latitude respectively on the
2-sphere. In our inhomogeneous coordinates $(z_1,z_2)$, the restriction of the
Fubini-Study metric to $\pi(L_{q,m})$ takes the form
\[  \frac{|dz_1|^2+|dz_2|^2}{1+\rho^2} -
\frac{|\bar z_1dz_1+\bar z_2dz_2|^2}{(1+\rho^2)^2}.\]
A computation shows that this reduces to the first fundamental form
\[ \frac{\rho^2\cos^2v(1+\rho^2\sin^2v)}{(1+\rho^2)^2}du^2 +
\frac{\rho^2}{1+\rho^2}dv^2.\]
We seek a profile curve $(f(v),g(v))$ which generates a surface of revolution
with the same first fundamental form, which requires setting
\[\begin{array}{l}
f(v) =\ds \frac{\rho\cos v\sqrt{1+\rho^2\sin^2v}}{1+\rho^2}\\[10pt]
f'(v)^2+g'(v)^2 =\ds \frac{\rho^2}{1+\rho^2}.
\end{array}\]
In order to plot Figure \ref{rot}, we determined $g(v)$ by numerical
integration.

\begin{figure}
  \includegraphics[scale=0.3]{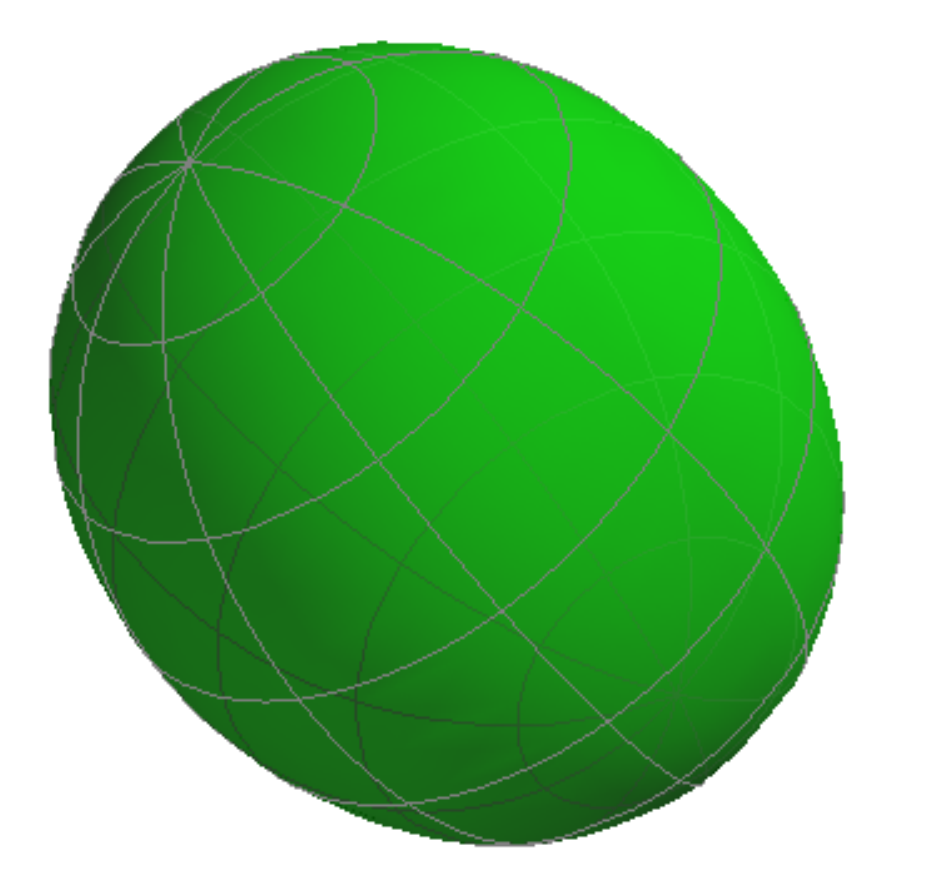}\quad\
  \includegraphics[scale=0.35]{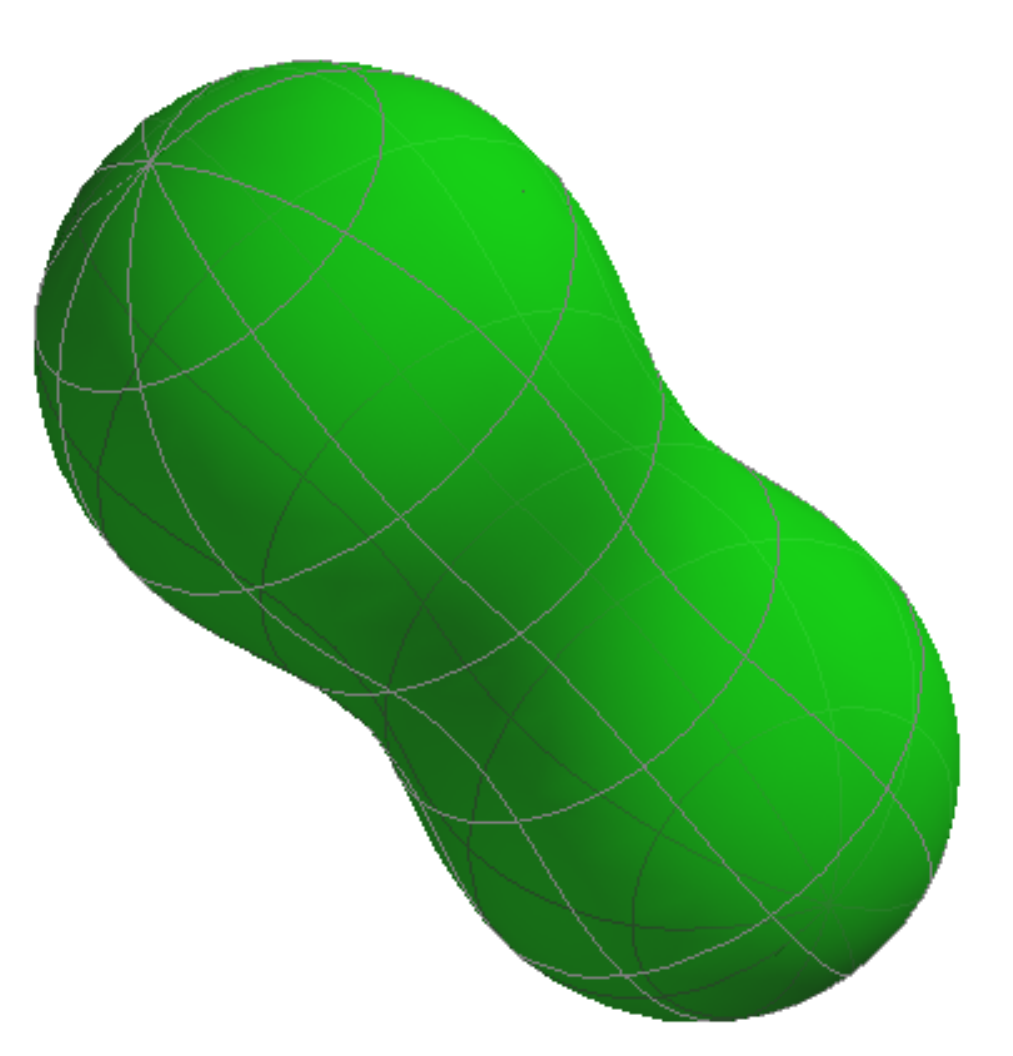}\quad
  \includegraphics[scale=0.4]{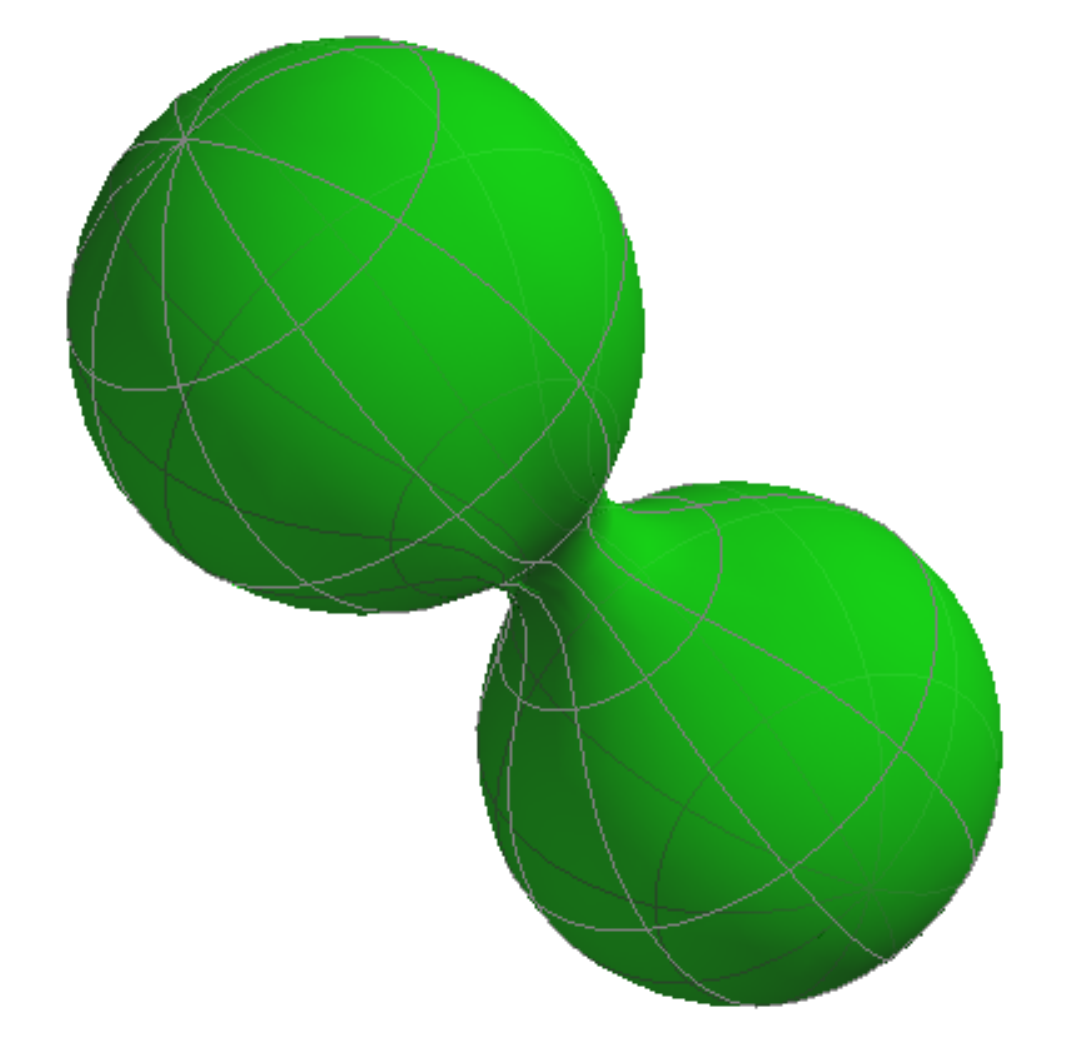}
  \vspace{20pt}
  \caption{Twistor images of the curve $L_{[1:0:0],[1:2\rho:0]}$ with $\rho=\frac12,2,8$}
  \label{rot}
\end{figure}

\subsection{Surfaces of bidegree (1,d)}
This interesting class of rational surfaces $S$ incorporates
increasingly wild examples as $d$ becomes large. On the one hand, $S$
is obtained by blowing up $k$ points in $\CP$, on the other it defines
a holomorphic $d$-fold branched cover over $\CPv$. 
Using Remark~\ref{formulac2}, we deduce that
\[ 9-k = c_1^2 = -d^2-d+8.\]
Such a surface will hit a generic twistor line in $d+1$ points, a number that
can be regarded as its `twistor degree', by comparison to the case of
$\PP^3\to S^4$.
  
The case $(1,2)$ is mentioned in the Introduction and corresponds to a del Pezzo
surface $P$ of degree 2 that double covers $\CP$ branched over a quartic. Here,
$P$ is the zero set of a polynomial lying in the summand $\Pp_{1,2}$ in the
notation of \eqref{Pp}.

\begin{example}
Let us consider the surface in the flag defined by
\[p_0(2\ell_0\ell_1+\ell_2^2)+p_1(2\ell_0\ell_2+\ell_1^2)+
p_2(2\ell_1\ell_2+\ell_0^2)=0.\]
This can be seen as the blow up of $\PP^{2\vee}$ in $7$ points in the
following way.  Such points can be obtained by imposing that, for a
fixed $p$ the previous equation is linearly dependent of $p\ell=0$,
i.e.
$$
\left(\begin{matrix}
2\ell_0\ell_1+\ell_2^2 & \ell_{0}\\
2\ell_0\ell_2+\ell_1^2 & \ell_{1}\\
2\ell_1\ell_2+\ell_0^2 & \ell_{2}
\end{matrix}\right)
$$
has rank one, that is if and only if
$$
\begin{cases}
\ell_2^3-\ell_0^3=0\\
\ell_0\ell_1^2+\ell_1\ell_2^2-2\ell_0^2\ell_2=0\\
\ell_0^2\ell_1+\ell_1^2\ell_2-2\ell_0\ell_2^2=0
\end{cases}
$$
Solving this system we get the following seven points.  There is a
trivial solution for $\ell_{0}=\ell_{1}=\ell_{2}$.  The other $6$
solutions arise in couples by imposing, $\ell_{2}=\eta^{k}\ell_{0}$ for $k=0,1,2$ and $\eta=e^{2\pi i/3}$. Therefore the seven
points are
$$
\begin{matrix}
[0:1:0], &[1:1:1], & [1:-2:1],  &
[1:\bar\eta:\eta], & [1:-2\bar\eta:\eta], &[1:\eta:\bar\eta], &
[1:-2\eta:\bar\eta].
\end{matrix}
$$

The surface can be seen as a double cover of $\PP^2$ branched over a
quartic. We identify the latter by studying the system
$$\begin{cases}
  p_0(2\ell_0\ell_1+\ell_2^2)+p_1(2\ell_0\ell_2+\ell_1^2)+
  p_2(2\ell_1\ell_2+\ell_0^2)=0\\
p_0l_0+p_1l_1+p_2l_2=0
\end{cases}$$

If $p_1\neq0$, multiplying by $p_1$ the first equation, we get
$$p_1p_2\ell_0^2+2p_1^2\ell_0\ell_2+p_0p_1\ell_2^2-(p_0\ell_0+p_2\ell_2)^2=0,$$
that is
$$(p_1p_2-p_0^2)\ell_0^2+2(p_1^2-p_0p_2)\ell_0\ell_2+(p_0p_1-p_2^2)\ell_2^2=0,$$
and the discriminant is:
$$ {\ts\frac14}\Delta = (p_1^2-p_0p_2)^2-
(p_1p_2-p_0^2)(p_0p_1-p_2^2)=
p_1^4-3p_0p_1^2p_2+p_1p_2^3+p_0^3p_1-p_0^2p_2^2
$$
If $p_{1}=0$ it is easy to check that we still obtain points in the previous quartic.
\end{example}

\begin{example}
Let us consider the surface in the flag defined by
\[p_{0}\ell_{1}^{2}+p_{1}\ell_{2}^{2}+p_{2}\ell_{0}^{2}=0.\]
This can be seen as the blow up of $\PP^{2\vee}$ in $7$ points in the
following way.  Such points can be obtained by imposing that, for a
fixed $p$ the previous equation is linearly dependent of $p\ell=0$,
i.e.
$$
\left(\begin{matrix}
\ell_{1}^{2} & \ell_{0}\\
\ell_{2}^{2} & \ell_{1}\\
\ell_{0}^{2} & \ell_{2}
\end{matrix}\right)
$$
has rank one, that is if and only if
$$
\begin{cases}
\ell_{1}^{3}-\ell_{0}\ell_{2}^{2}=0\\
\ell_{0}^{3}-\ell_{2}\ell_{1}^{2}=0\\
\ell_{2}^{3}-\ell_{1}\ell_{0}^{2}=0
\end{cases}
$$
Solving this system we get the following seven points for $k=0,\dots, 6$
$$
[\zeta^{3k}:\zeta^{k}:1],
$$
where $\zeta$ is any seventh root of one.

The surface can be seen as a double cover of $\PP^2$ branched over a
quartic. We identify the latter by studying the system
$$\begin{cases}
 p_{0}\ell_{1}^{2}+p_{1}\ell_{2}^{2}+p_{2}\ell_{0}^{2}=0\\
p_0l_0+p_1l_1+p_2l_2=0
\end{cases}$$

If $p_0\neq0$, multiplying by $p_0^{2}$ the first equation, we get
$$p_{0}^{3}\ell_{1}^{2}+p_{0}^{2}p_{1}\ell_{2}^{2}+p_{0}^{2}p_{2}\ell_{0}^{2}=0$$
that is
$$(p_{0}^{3}+p_{1}^{2}p_{2})\ell_{1}^{2}+2p_{1}p_{2}^{2}\ell_{1}\ell_{2}+(p_{0}^{2}p_{1}+p_{2}^{3})\ell_{2}^{2}=0,$$
and the discriminant is:
$$ {\ts\frac14}\Delta = -p_{0}^{2}(p_{0}^{3}p_{1}+p_{1}^{3}p_{2}+p_{2}^{3}p_{0}).
$$
If $p_{0}=0$ it is easy to check that we still obtain points in the quartic defined by $p_{0}^{3}p_{1}+p_{1}^{3}p_{2}+p_{2}^{3}p_{0}=0$.
\end{example}

\section{Twistor projections of smooth toric surfaces}\label{Tw1}

Having studied the twistor images of curves and (at the start of
Section \ref{Twn}) surfaces of bidegree $(1,0)$ or $(0,1)$, we turn
attention to the next simplest case. Our general goal will be to
determine the subsets of $\CP$ defined by

\begin{definition}\label{def-discriminant}
Let $S$ be any element of $|\Of(a,b)|$. Its \textit{twistor discriminant locus}
$\Dd(S)$ is the union
\[\Dd=\Dd_0\cup \Rr,\]
where $\Dd_0=\{q\in\CP\mid \pi^{-1}(q)\subset S\}$ and $\Rr$ is the branch
locus of $\pi_{|S}:S\to\CP$ consisting of points $p\in \CP$ for which
$|\pi^{-1}(p)|<a+b$.
\end{definition}

We shall determine the discriminant locus for smooth surfaces in $\FF$ of
bidegree $(1,1)$. Any such surface is defined by the canonical form
\begin{equation}\label{Aone}
A=A_1=\left(\begin{array}{ccc}
  0&0&0\\
  0&1&0\\
  0&0&\la
\end{array}\right),\qquad \la\in\{0,1\},
\end{equation}
from Lemma \ref{matricesA}. In order to make sense of twistorial
properties, and since we are ultimately interested in the metric
geometry of $\Dd$, we work up to unitary equivalence.

Recall that $S_A,S_{A'}$ are unitarily equivalent if and only if there exists
$X$ in $\SU(3)$ (equivalently, in $U(3)$ such that
\[A' = X^*\!AX = X^{-1}\!AX.\]
Since \eqref{Aone} commutes with elements of the standard maximal torus $T^2$
of $\SU(3)$, our constructions will be $T^2$-invariant. This is not surprising,
since it is well known that the del Pezzo surfaces arising in Proposition
\ref{DelP} are in fact toric. We first treat the case in which $S_A$ is real
(meaning that $j(S)=S$), whose conclusion will motivate the non-real case.

\subsection{Real (1,1)-surfaces}\label{real11}
Here we focus on the $j$-invariant case, for the moment without assuming
smoothness. It follows from the definitions \eqref{SA} and \eqref{jjj} that
$S_A$ is $j$-invariant if and only if $A=A^*$, i.e.\ $A$ is Hermitian. Such a
matrix has real eigenvalues, is diagonalizable over $\SU(3)$, and $S_A$ can be
represented by a real diagonal matrix. Proposition
\ref{classification-singular} implies that $S_A$ is either smooth and of type
$A_1$ in Lemma \ref{matricesA}, or reducible and of type $A_2$.

The reducible surfaces correspond to matrices $A$ with two distinct
eigenvalues; any one has the form $\qH\cup H_{q^*}$, is singular on the twistor
line $L_{q,q^*}$, and contains no other twistor line. It follows that any two
are unitarily equivalent, so little more needs to be written.

In the smooth case, we can impose \eqref{Aone} with $\lambda\in(1,2]$. This
restriction arises from Remark \ref{Xratio}, and ensures that the unitary
equivalence class of $S_A$ is uniquely determined by a point of the interval
$(1,2]$.

\begin{theorem}\label{nostro}
  Let $S$ be a smooth $j$-invariant surface of bidegree $(1,1)$. The twistor
  projection $\pi$ restricted to $S$ is a degree 2 cover of $\CP$ without
  ramification, i.e. $\Rr=\es$. Moreover, $S$ contains infinitely many twistor
  fibers and $\Dd_0$ is parametrized by a circle that is the orbit of a maximal
  torus in $SU(3)$.
  \end{theorem}  
  
\begin{proof}
Fix $q\in \CP$. Recall that its twistor fiber is
\[\pi^{-1}(q) = L_{q,q^*} =
\{(p,\ell)\mid p\ell=0,\ q\ell=0,\ pq^*=0\} = \qH\cap H_{q^*}.\]
Now
\[ S_A\cap \qH = \{(p,\ell)\mid p\ell=0,\ pA\ell=0,\ q\ell=0\}\]
is non-empty if and only if
\begin{equation}\label{eq-det}
\det(p\mid pA\mid q) = 0.
\end{equation}
This equation can be written as $p\,Cp^\vee=0$, where
\[C=\left(\!\begin{array}{ccc}
    0 & q_2 & -\la q_1\\
    q_2 & 0 & (\la-1)q_0\\
   -\la q_1 & (\la-1)q_0 & 0
  \end{array}\!\right), \]
which defines a conic $\Cc$. Applying the involution $j$ shows that
\[ \pi(S_A\cap\qH)=\Cc=\pi_2^*(S_A\cap H_{q^*}).\]
The remaining equation
\begin{equation}\label{eq-retta}
    pq^* = 0
\end{equation}
asserts that $p$ lies on the line $q^*$, so there will be two such points on
$C$ for generic $q$. If $C$ contains the twistor fiber $\pi^{-1}(q)$ then $C$
must be reducible and
\[ 0=\det C = 2\la(\la-1)q_0q_1q_2,\]
so that one of $q_0,q_1,q_2$ vanishes. We obtain the following three equations
for $\Cc$:
\[\left\{\begin{array}{ll}
p_0(q_2p_1- q_1\la p_2)=0 & \hbox{if \ $q_0=0$}\\
p_1(q_2p_0+ q_0(\la-1)p_2)=0 & \hbox{if \ $q_1=0$}\\
p_2(-q_1\la p_0+ q_0(\la-1)p_1)=0 & \hbox{if \ $q_2=0$}
\end{array}\right.\]
Combining these with \eqref{eq-retta},
\begin{equation}\label{qqq}
\left\{\begin{array}{lcl}
q_0=0 & \hbox{and} & |q_1|^2\la+|q_2|^2=0\\
q_1=0 & \hbox{and} & |q_2|^2-|q_0|^2(\la-1)=0\\
q_2=0 & \hbox{and} & |q_1|^2\la+|q_0|^2(\la-1)=0.
\end{array}\right.
\end{equation}
However, our assumption that $\la>1$ rules out all but the second possibility,
so $q_1=0$ and $|q_2/q_0|=\sqrt{\la-1}$. This is indeed a circular orbit of
$\TT$ with a 1-dimensional stabilizer.

We shall now determine the branch locus $\Rr$. If $q\in\Rr$ then clearly the
line \eqref{eq-retta} is tangent to $C$. If $q_0q_1q_2=0$ then $C$ is reducible
and tangency would imply $q\in\Dd_0$. We can therefore assume that
$q_0q_1q_2\ne0$. In general, tangency implies that
\begin{equation}\label{cond-tang}
  0 = q^*\times(Cp^\vee) = (q^*\times C)\,p^\vee,
\end{equation}
in which the cross product $q^*\times C$ is a $3\times3$ matrix computed from
$C$ column by column. It follows that
\[ \left(\begin{array}{ccc} 
  \la|q_1|^2+|q_2|^2 & -(\la-1)q_0\bar q_1 & (\la-1)q_0\bar q_2\\
  -\la q_1\bar q_0 & (\la-1)|q_0|^2- |q_2|^2 & \la q_1\bar q_2\\
  -q_2\bar q_0 & q_2\bar q_1 & -(\la-1)|q_0|^2-\la|q_1|^2
  \vphantom{\lower6pt\hbox{o}}\\\hline
  \bar q_0 & \bar q_1 & \bar q_2
\end{array}\right)p^\vee = 0.\]
The $4\times 3$ matrix incorporates $q^*\times C$ and \eqref{eq-retta}, and
must therefore have rank at most 2. Note that $q^*\cdot(q^*\times C)=0$, so
$\det(q^*\times C)=0$. Since $q_0q_1q_2\ne0$, the remaining $3\times3$ minors
each tell us that
\[(|q_0|^2(\la-1)+|q_1|^2\la-|q_2|^2)^2+4|q_0|^2|q_2|^2(\la-1)=0.\]
Since $\la>0$ and $q_0q_2\ne0$, we conclude that $\Rr=\es$.
\end{proof}

\begin{remark}
Theorem \ref{nostro} gives a clear expectation of how the discriminant locus
behaves for a smooth surface $S_A$ with $A$ diagonalizable, and therefore
toric. Consider the moment mapping
\[\begin{array}{lrcl}
\mu\colon& \CP &\lra& \De\subset\RR^3\\
&[q_0:q_1:q_2] &\mapsto& (|q_0|^2,|q_1|^2,|q_2|^2)/\|q\|^2
\end{array}\]
corresponding to the action of $\TT$ on $\CP$, equipped with its Fubini-Study
symplectic 2-form. Its image is a 2-simplex $\De$, and the union of lines
$q_0q_1q_2=0$ is mapped onto the boundary $\partial\De$.

If $A$ is Hermitian, its discriminant circle is mapped to a point of
$\partial\De$. This is a midpoint of an edge of $\partial\De$ if
$\la\in\{-1,\frac12,2\}$, in which case the circle is maximal. Such circles
played a key role in the configuration of equidistant points in $\CP$
\cite{HS}. The prohibited values $\la\in\{0,1,\infty\}$ correspond to the
vertices of $\De$ for which $S_A$ is reducible and the discriminant locus
becomes a single point. We shall see that if $A$ is diagonalizable, but not
Hermitian, its discrimant locus is a
smooth 2-torus that maps to an interior point of $\De$.
  \end{remark}

\begin{corollary}\label{cordisco}
Let $S$ be as in the hypothesis of Theorem \ref{nostro}. Then $\CP\sm\Dd_0$ is
the maximal domain of an orthogonal complex structure $J_S$.
\end{corollary}

\begin{proof}
First, we observe that $S_A\sm\pi^{-1}(\Dd_0)$ has two connected components.
This follows directly from the generalized Jordan Curve Theorem~\cite[Theorem
  VI.8.8]{bredon}, given that $S_A$ is the blow-up of $\CP$ at three points.
The components are interchanged by $j$. Either one projects bijectively
onto $\CP\sm\Dd_0$ and induces the OCS $J_S$.
\end{proof}

\subsection{Toric non-real (1,1)-surfaces}
In this subsection, we shall describe the twistor discriminant locus for a
surface $S_A$ with $A$ in the diagonal form \eqref{Aone} with
$\la\in\CC\sm\RR$. This condition on $\la$ ensures that $S_A$ is smooth but not
$j$-invariant.

\begin{theorem}\label{diagonaltheorem}
Let $S_A$ be the surface defined by \eqref{Aone} with $\la\in\CC\sm\RR$. Then
the twistor projection $\pi$ restricted to $S$ is a degree 2 cover of $\CP$
whose branch locus $\Rr\subset\CP$ is the zero set of
\[ R(q) = \{(|q_0|^2(\la-1)+|q_1|^2\la+|q_2|^2)^2-4|q_0|^2|q_2|^2(\la-1)=0\}.\]
Moreover, $S$ contains no twistor fibers, i.e.\ $\Dd_0=\es$.
\end{theorem}

\begin{proof}
Inevitably, we follow the same first steps of the proof of
Theorem~\ref{nostro}.

The surface $S_A$ contains the fiber $\pi^{-1}(q)$ if and only if $q_0=0$ and
$|q_1|^2\la+|q_2|^2=0$ or $q_1=0$ and $|q_2|^2-|q_0|^2(\la-1)=0$ or $q_2=0$ and
$|q_1|^2\la+|q_0|^2(\la-1)=0$.  But in all three cases, as $\la_1\ne0$, we
have no solution $q\in\CP$.

We now want to study the branch locus. As in the proof of Theorem \ref{nostro},
if $q_0q_1q_2=0$ there is no branch locus.  Hence for the remainder, assume
$q_0q_1q_2\ne0$. Again, following the proof of Theorem \ref{nostro}, we obtain
that the branch locus of the map $\pi_{|S_A}$ is given $R(q)=0$ where $R(q)$ is
defined as in the statement of the theorem.
\end{proof}

In the following corollary we give a geometric description of the set $\Rr$.

\begin{corollary}
Let $S_A$ be a $(1,1)$-surface as in
Theorem~\ref{diagonaltheorem}. Let $\alpha=\alpha_0+i\alpha_1$ be a
square root of $\lambda -1$ and $\beta=\beta_0+i\beta_1$ a square root
of $\la$, such that $\alpha_1\beta_0>0$ and
$\frac{\alpha_0\beta_0+\alpha_1\beta_1}{\alpha_1}>0$. Then the branch
locus of $\pi_{|S_A}$ is
$$
\Rr=\{q\in\CP\mid |q_0|\alpha-i|q_1|\beta-|q_2|=0\}.
$$
In particular, the twistor discriminant locus of $S_{A}$ is a smooth
$2$-dimensional real torus.
\end{corollary}

\begin{proof}
First of all notice that, since $\la\in\CC\sm\RR$, as $\alpha^2=\la-1$ and
$\beta^2={\la}$, then $\alpha_0\alpha_1\ne0$ and
$\beta_0\beta_1\ne0$. Moreover up to changes of signs of $\alpha$ and $\beta$
we can always reduce to the assumptions of the statement. Given that,
the proof is smoothly obtained thanks to the following computation:
\begin{align*}
R(q)&=(|q_0|^2(\la-1)+|q_1|^2\la+|q_2|^2)^2-4|q_0|^2|q_2|^2(\la-1)\\
&=(|q_0|^2(\la-1)+|q_1|^2\la+|q_2|^2-2|q_0||q_2|\alpha)(|q_0|^2(\la-1)+|q_1|^2\la+|q_2|^2+2|q_0||q_2|\alpha)\\
&=((|q_0|\alpha-|q_2|)^2+|q_1|^2\la)((|q_0|\alpha+|q_2|)^2+|q_1|^2\la)\\
&=(|q_0|\alpha-i|q_1|\beta-|q_2|)(|q_0|\alpha+i|q_1|\beta-|q_2|)(|q_0|\alpha-i|q_1|\beta+|q_2|)(|q_0|\alpha+i|q_1|\beta+|q_2|).
\end{align*}
We claim that only the first term of this factorization vanishes. 
Indeed, 
setting
\begin{align*}
R_{--}(q)=|q_0|\alpha-i|q_1|\beta-|q_2|,\quad&\quad R_{+-}(q)=|q_0|\alpha+i|q_1|\beta-|q_2|,\\
R_{-+}(q)=|q_0|\alpha-i|q_1|\beta+|q_2|,\quad&\quad R_{++}(q)=|q_0|\alpha+i|q_1|\beta+|q_2|,
\end{align*}
And splitting into real and imaginary part, we get that $R_{**}(q)=0$ ($*=+,-$) if and only if
$$
\begin{cases}
|q_0|\alpha_0\mp i|q_1|\beta_1\pm|q_2|=0\\
|q_0|\alpha_1\pm i|q_1|\beta_0=0
\end{cases}
$$
By looking at the second equation, we get that $R_{-*}(q)=0$ admits solutions and $R_{+*}(q)=0$ doesn't, since $\alpha_1\beta_0>0$.
Moreover, we have that $|q_0|=|q_1|\frac{\beta_0}{\alpha_1}$ and hence, the first equation can be written as
$$
|q_1|\frac{\alpha_0\beta_0+\alpha_1\beta_1}{\alpha_1}\pm|q_2|=0.
$$ Therefore, $R_{--}(q)=0$ admits solutions and $R_{-+}(q)=0$ doesn't, since
$\frac{\alpha_0\beta_0+\alpha_1\beta_1}{\alpha_1}>0$.
\end{proof}

\begin{remark}
In the previous corollary, we saw by direct computations that the set
$\Rr$ is a torus. This fact can also be seen because $\Rr$ is
invariant under the torus action
$(\vartheta_1,\vartheta_2)\cdot[q_0:q_1:q_2] =
[q_0:e^{i\vartheta_1}q_1,e^{i\vartheta_2}q_2]$, for any
$\vartheta_1,\vartheta_2\in[0,2\pi)$.
\end{remark}

Theorems~\ref{nostro} and~\ref{diagonaltheorem} can be considered as
\textit{toy models} to study the most general case. We will see now that the
general analysis of the twistor discriminant locus of a smooth $(1,1)$ surface
is highly non-trivial.  In particular we decided to split the analysis of the
twistor discriminant locus into two parts: first we study the set $\Dd_0$ of
twistor fibers contained in a smooth $(1,1)$ surface, then we give Cartesian
equations for the set $\Rr$.

\section{Counting twistor lines in a surface}\label{infinity}

In this section, we give a first upper bound on the number of twistor lines
contained in a surface $S$ which is not $j$-invariant, based on a Bezo\'ut type
method. For higher bidegrees, the bounds are not likely to be optimal. In
particular, we shall prove in Corollary \ref{2 rette} that a non $j$-invariant
$(1,1)$ surface contains at most two twistor lines, and in this case the bound
is actually sharp.

First of all, notice that if $S\in|\Of(a,b)|$, then $j(S)\in|\Of(b,a)|$.

The next result gives a first bound on the number of $(1,1)$-curves contained
in a non $j$-invariant surface $S$.

\begin{proposition}\label{number-lines}
Fix $a,b\ge 0$ and an integral $S\in |\Of(a,b)|$. If $j(S)\ne S$, then
$S\cap j(S)$ contains at most $a^2+ab+b^2$ integral curves of bidegree $(1,1)$.
\end{proposition}

\begin{proof}
Since $S$ is integral and $j(S)\ne S$, then $S\cap j(S)$ is a finite union of
(possibly non-reduced) curves, that is $$S\cap j(S) = \sum m_iC_i$$ with $m_i$
positive integers, and $C_i$ an irreducible curve with bidegree $(E_i,F_i)$ for
any $i$.  Recall that $E_i=\deg(\pi_1(C_i))\deg(\pi_{1|C_i})$ and
$F_i=\deg(\pi_2(C_i))\deg(\pi_{2|C_i})$.  Then $$\sum
m_i\deg(\pi_{1|C_i})=\Of(a,b)\cd \Of(b,a)\cd \Of(1,0)$$ as well as
$$\sum m_i\deg(\pi_{2|C_i})=\Of(a,b)\cd \Of(b,a)\cd \Of(0,1).$$
But we have. 
\begin{multline*}
  \Of(a,b)\cd \Of(b,a)\cd \Of(1,0) =
  (a\Of(1,0)+b\Of(0,1))\cd(b\Of(1,0)+a\Of(0,1))\cd \Of(1,0)\\
  =(ab\Of(1,0)\cd \Of(1,0)+a^2\Of(1,0)\Of(0,1)+
  b^2\Of(1,0)\Of(0,1)+ab\Of(0,1)\Of(0,1))\cd \Of(1,0),
\end{multline*}
and
\begin{multline*}
  \Of(a,b)\cd \Of(b,a)\cd \Of(0,1) =
  (a\Of(1,0)+b\Of(0,1))\cd(b\Of(1,0)+a\Of(0,1))\cd \Of(0,1)\\
  =(ab\Of(1,0)\cd \Of(1,0)+a^2\Of(1,0)\Of(0,1)+
  b^2\Of(1,0)\Of(0,1)+ab\Of(0,1)\Of(0,1))\cd \Of(0,1).
\end{multline*}

Therefore, by Proposition \ref{intersezione}, we have
\[\Of(a,b)\cd \Of(b,a)\cd \Of(1,0) =
\Of(a,b)\cd \Of(b,a)\cd \Of(0,1)=a^2+ab+b^2.\]

Hence $S$ contains at most $a^2+ab+b^2$ integral curves of bidegree $(1,1)$
(just setting $E_i=F_i=1$).
\end{proof}

By means of Lemma~\ref{p2}, we are able to prove the following
proposition that will be used to refine the previous bound.

\begin{proposition}\label{oo1}
  Let $a,b,c,d$ be positive integers, $S\in |\OF(a,b)|$ and $D\in |\Oo_S(c,d)|$.
  Then $H^0(\Oo_S)=H^0(\Oo_D)=\CC$, i.e. every holomorphic function on
$D$ is constant. Thus $S$ and $D$ are connected.
\end{proposition}

\begin{proof}
Consider the exact sequence
\begin{equation}\label{eqoo1}
0 \to \OF(-a,-b)\to \OF\to \Oo_S\to 0
\end{equation}
Since, by Lemma~\ref{p2}, $H^0(\OF) =\CC$ and $H^1(\OF(-a,-b))=0$, we
get $H^0(\Oo_S)=\CC$. Tensoring \eqref{eqoo1} by $\OF(-c,-d)$ and
using Lemma~\ref{p2} we also get that $H^1(\Oo_S(-c,-d))=0$.

Hence by the exact sequence
\begin{equation}\label{eqoo2}
0 \to \Oo_S(-c,-d)\to \Oo_S\to \Oo_D\to 0
\end{equation}
we also get $H^0(\Oo_D) =\CC$.
\end{proof}

Recall that a twistor line is a $j$-invariant integral curve of bidegree
$(1,1)$. A surface of bidegree $(1,0)$ or $(0,1)$ contains exactly one twistor
line. As a consequence of the previous two results, we get the following bounds
on the number of twistor lines contained in a surface which is not
$j$-invariant.

\begin{corollary}\label{number-twistorlines}
Fix $a,b> 0$ and an integral $S\in |\Of(a,b)|$. If $S\ne j(S)$, then $S$
contains at most $a^2+ab+b^2-1$ twistor lines.
\end{corollary}

\begin{proof}
Notice that if $L$ is a twistor line contained in $S$, then $L\in S\cap
j(S)$. Hence by Proposition \ref{number-lines}, we know that the number of
twistor lines contained in $S$ is at most $a^2+ab+b^2$.  We prove now that $S$
does not contain exactly $a^2+ab+b^2$ twistor lines. Assume that this is the
case.  Then we would have that $S\cap j(S)$ is the union of $a^2+ab+b^2\ge 2$
pairwise disjoint curves.  In particular $S\cap j(S)$ is not connected.  But by
Proposition \ref{oo1}, noting that $D=S\cap j(S)\in|\Oo_S(b,a)|$, we have that
$S\cap j(S)$ is connected.  This gives a contradiction and concludes the proof.
\end{proof}

\begin{corollary}\label{2 rette}
Let $S\in|\Of(1,1)|$. If $S\ne j(S)$, then $S$ contains at most $2$ twistor
lines.
\end{corollary}

\subsection{Configurations of lines}
\label{configuration-sec}
The bound from Corollary~\ref{number-twistorlines} demonstrates an interesting
phenomenon. As a non $j$-invariant $S\in|\Of(1,1)|$ contains at most $2$ twistor
lines, then, by Proposition~\ref{number-lines} the intersection $S\cap j(S)$
contains at most just one integral $(1,1)$ curve $C$ non $j$-invariant, but
this is impossible, otherwise $S\cap j(S)$ would contains $j(C)$ as well
contradicting the bound (see Remark~\ref{remarkintersectionj}). This means that
in this case, as $S\cap j(S)$ is connected, there should be another
$j$-invariant set of curves of possible bidegrees $(1,0)$ or $(0,1)$ that
connect the two twistor fibers. The next examples illustrate exactly this
situation.

\begin{example}\label{primo ex}
Consider the $(1,1)$ surface $S=S_A$ in $\FF$ defined by the matrix
$$
A=\left(\begin{matrix}
0&3&0\\
0&1&0\\
0&0&2
\end{matrix}
\right).
$$
Then $A$ defines a smooth non $j$-invariant $(1,1)$-surface, and 
$$
S=\{(p,\ell)\in\FF\mid (3\ell_0+\ell_1)p_1+2\ell_2p_2=0\},\quad
j(S)=\{(p,\ell)\in\FF\mid (3p_0+p_1)\ell_1+2p_2\ell_2=0\}.
$$
Therefore the intersection $S\cap j(S)$ in $\FF$ is given by the system
\begin{equation}\label{systemexample}
\begin{cases}
\ell_0p_0+\ell_1p_1+\ell_2p_2=0\\
(3\ell_0+\ell_1)p_1+2\ell_2p_2=0\\
(3p_0+p_1)\ell_1+2p_2\ell_2=0
\end{cases}
\end{equation}

By interpreting this as a linear system with $l$ unknowns and $p$ as
parameters, we get that it admits a solution in $\PP^\vee$ if and
only if the determinant of the associated matrix vanishes, i.e.\ when
\begin{equation}\label{firstdet}
-3p_2(p_0-p_1)(2p_0-p_1)=0.
\end{equation}
By analyzing the factors of \eqref{firstdet} we obtain that
\begin{itemize}
\item if $p_2=0$, then $\ell_1=0$ and $\ell_0=0$; hence, in this case
  \eqref{systemexample} has the following set of solutions:
  $\{(p,\ell)=([*:*:0],[0:0:1])\}$, that is $\pi_2^{-1}([0,0,1])$;
\item if $p_0=0=p_1$, then we obtain the solution set
  $\pi_2^{-1}([0,0,1])$;
\item if $p_1=p_0\ne0$, then we get that
  $\ell_0-\ell_1=0=2\ell_0p_0+\ell_2p_2$, but these determine the
  twistor fiber $L_{[1:-1:0],[1:-1:0]}$;
\item if $p_1=2p_0\ne0$, in analogy with the previous case, we get
  $2\ell_0-\ell_1=0=5\ell_0p_0+\ell_2p_2=0$, which determine the
  twistor fiber $L_{[2:-1:0],[2:-1:0]}$.
\end{itemize}

Oberve that $j(\pi_2^{-1}([0,0,1]))=\pi_1^{-1}([0,0,1])$, so the set of
solutions of \eqref{systemexample} is indeed $j$-invariant. Of course,
$L_{[1:-1:0],[1:-1:0]}$ and $L_{[2:-1:0],[2:-1:0]}$ are disjoint but these two lines should
intersect $\pi_2^{-1}([0,0,1])\cup\pi_1^{-1}([0,0,1])$. We now
compute these mutual intersections.  First, recall from
Remark~\ref{remarkintersectionj} that
$\pi_2^{-1}([0,0,1])\cap\pi_1^{-1}([0,0,1])=\es$.  By direct inspection we get
the following intersections
$$
\begin{cases}
A=L_{[1:-1:0],[1:-1:0]}\cap \pi_2^{-1}([0,0,1])=([1:1:0],[0:0:1])\\
B=L_{[1:-1:0],[1:-1:0]}\cap \pi_1^{-1}([0,0,1])=([0:0:1],[1:1:0])\\
C=L_{[2:-1:0],[2:-1:0]}\cap \pi_2^{-1}([0,0,1])=([1:2:0],[0:0:1])\\
D=L_{[2:-1:0],[2:-1:0]}\cap \pi_1^{-1}([0,0,1])=([0:0:1],[1:2:0])
\end{cases}
$$

\begin{figure}
  \includegraphics[scale=0.6]{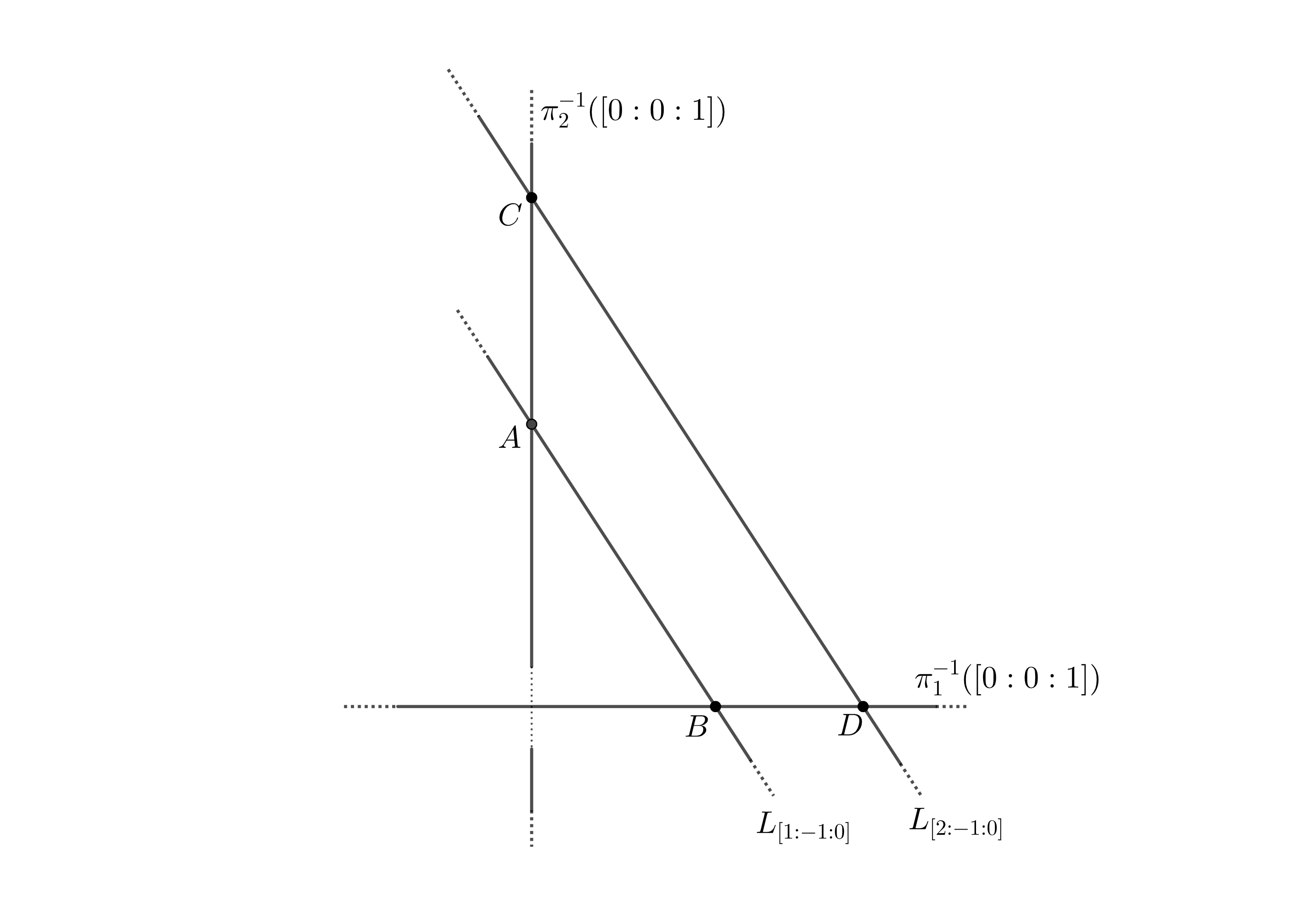}
  \vspace{-20pt}
\caption{Intersection of $S$ and $j(S)$. Notice how $S\cap j(S)$ is a
  connected set of two $(1,1)$-curves and two canonical (skew)
  fibers.}
\end{figure}
\end{example}

We now give two more examples of smooth $(1,1)$-surfaces, the first contains just one
twistor fiber, and the second none.

\begin{example}\label{secondo ex}
Consider now the $(1,1)$ surface $S$ in $\FF$ defined by the matrix
$$
A=\left(\begin{matrix}
0&2\sqrt2&0\\
0&1&0\\
0&0&2
\end{matrix}
\right).
$$
With analogous computations as in the previous example, we can see
that the intersection $S\cap j(S)$ is given by the union of the
twistor line {$L_{[1:-\sqrt2:0],[1:-\sqrt2:0]}$} and the two fibers
$\pi_1^{-1}([0,0,1])$ and $\pi_2^{-1}([0,0,1])$.
\end{example}

\begin{example}\label{terzo ex}
Consider finally the $(1,1)$ surface $S$ defined by
$$
A=\left(\begin{matrix}
0&2&0\\
0&1&0\\
0&0&2
\end{matrix}
\right).
$$
In this case we get that the intersection $S\cap j(S)$ is given by the
union of two lines {$L_1=L_{[1+i:-1:0],[1+i:-1:0]}$} and
{$L_2=L_{[1-i:-1:0],[1-i:-1:0]}$} and the two fibers $\pi_1^{-1}([0,0,1])$
and $\pi_2^{-1}([0,0,1])$. Note that $j(L_1)=L_2$ and there are no
twistor lines contained in $S$.
\end{example}

We conclude this section with the following easy consequence of
Proposition~\ref{number-lines}.
\begin{corollary}\label{a=b}
Let $S\in|\Of(a,b)|$. If $S$ contains infinitely many twistor lines,
then $S$ is $j$-invariant and hence $a=b$.
\end{corollary}

\begin{remark}
Subsequent to the first version of this paper, the existence of integral
$j$-invariant surfaces of type $(a,a)$ containing infinitely twistor
lines for arbitrary $a > 1$ was established in the final section of
\cite{altavillaballicobrambilla}. Only if $a=1$ the surfaces
are smooth; the singular set always has dimension one if $a\ge2$.
\end{remark}

\section{Twistor projections of some non-real surfaces}\label{ramf}

In Section~\ref{infinity}, we gave upper bounds on the number of
twistor fibers contained in a non $j$-invariant integral surface and
showed with an example, that such bound is attained in the case of
$(1,1)$-surfaces.  In this section we give an explicit proof that any
non $j$ invariant smooth $(1,1)$ surface contains at most 2 twistor
lines.

{ We will need the following variations on the theme of solving quadratic
  equations:

\begin{lemma}\label{zf}
  Let $f=f_0+if_1\in\CC$. If $f_0+f_1^2\le\frac14$ then the equation
\[|z|^2+z+f=0\]
has one or two solutions
\[ z = -\tfrac12\pm \tfrac12\sqrt{1-4f_0-4f_1^2}-if_1.\]
If $f_0+f_1^2>\frac14$ then it has no solution.
\end{lemma}

If we set $z=x+iy$ then obviously $y=-f_1$ and the lemma follows immediately
from the usual quadratic formula applied to $x=\Re z$.

\begin{corollary}\label{zef}
Let $e$ be a non-zero complex number, and consider the equation  
\begin{equation}\label{cxeq}
|z|^2+ez+f=0.
\end{equation}
Set $\De=|e|^4-4f_0|e|^2-4f_1^2$. Then \eqref{cxeq} admits a solution if and
only if $\De\ge0$, in which case
\[ z = -\Big(\tfrac12|e|^2\pm \tfrac12\sqrt\De +if_1\Big)e^{-1}.\]
\end{corollary}  

\begin{proof}
Set $Z=ez/|e|^2$ and $F=f/|e|^2$. Dividing \eqref{cxeq} by $|e|^2$, we obtain
\[ |Z|^2+Z+F=0.\]
Lemma \ref{zf} gives 
\[ Z = -\tfrac12\pm\tfrac12\sqrt{1-4F_0-4F_1^2}-iF_1,\]
provided the square root is real. The statement follows.
\end{proof}}

\subsection{Twistor lines again}
Corollary \ref{zef} enables us to state the main result of this section in
which we establish real algebraic conditions on the coefficients of the matrix
$A$ for which the surface $S_A$ contains $2$, $1$ or no twistor fibers.

\begin{theorem}\label{twistorfibers}
Let $S=S_{A}$ be a smooth surface of bidegree $(1,1)$ such that $S\neq j(S)$
and the matrix $A$ has the form
$$A=\left(\begin{array}{ccc}
  0&a&b\\
  0&1&c\\
  0&0&\la
  \end{array}\right),$$
where $\la=\la_0+i\la_1\in\CC\sm\{0\}$ and $[a:b:c]\in\CP$. If one of the
following mutually exclusive conditions is satisfied, then $S$ contains exactly
one twistor fiber:
\begin{enumerate}
\item[(i)] $b\ne0\ne c$ and $|b|^2(1-\la)-|c|^2\la-a\bar b c=0$;
\item[(ii)] $a\ne0\ne c$ and $|a|^2(1-\la)+|c|^2{+} a \bar b c=0$;
\item[(iii)] $a\ne0\ne b$ and $|a|^2\la +|b|^2-a\bar b c=0$;
\item[(iv)] $b= 0=c$ and 
{  $|a|^4-4|a|^2(|\la|^2-\la_0)-4\la_1^2=0$};
\item[(v)] $a=0= c$ and {
$|b|^4-4(1-\la_0)|b|^2-4\la_1^2=0$;}
\item[(vi)] $a= 0= b$ and {
$|c|^4-4\la_0|c|^2-4\la_1^2=0$.}
\end{enumerate}
If one of the following mutually exclusive conditions is satisfied, then $S$
contains exactly two twistor fibers:
\begin{enumerate}
\item[(vii)] $b= 0=c$ and { $|a|^4-4|a|^2(|\la|^2-\la_0)-4\la_1^2>0$} ;
\item[(viii)] $a=0=c$ and { $|b|^4-4(1-\la_0)|b|^2-4\la_1^2>0$};
\item[(ix)] $a= 0= b$ and { $|c|^4-4\la_0|c|^2-4\la_1^2>0$.}
\end{enumerate}
\end{theorem}

\begin{proof}
  Let $S_A$ be the surface defined by the matrix
$$A=\left(\begin{array}{ccc}
  0&a&b\\
  0&1&c\\
  0&0&\la
  \end{array}\right)$$
with $a,b,c,\la\in\CC$, and $\la=\la_0+i\la_1\neq0,1$.

As in the proof of Theorem~\ref{nostro}  we impose
  \begin{equation}\label{condizione-per-conica}
\det(p\mid pA\mid q) = 0,
  \end{equation}
  and we get the conic $C$
  $$
(q_2)p_0p_1+(-\la q_1+cq_2)p_0p_2+  (-aq_2)p_1^2+(-cq_0+bq_1)p_2^2+((\la-1)q_0+aq_1-bq_2)p_1p_2=0.
  $$
Intersecting it with the line $pq^*=0$, we get generically two
intersection points which give the two pre-images of $q$. The conic $C$ is
singular when the matrix
$$H=\left(\begin{array}{ccc}
  0&q_2&-\la q_1+cq_2\\
  q_2&-2aq_2&(\la-1)q_0+aq_1-bq_2\\
  -\la q_1+cq_2&(\la-1)q_0+aq_1-bq_2&2(-cq_0+bq_1)
 \end{array}\right)$$
has vanishing determinant, i.e.\
 $$
 2 q_2 ((1-\la)q_1+cq_2)(\la q_0-a\la q_1+(ac-b)q_2)=0.
$$
This equation describes the union of three lines 
\begin{align*}
r_1:\qquad &q_2=0,\\
r_2:\qquad &(1-\la)q_1+cq_2=0,\\
r_{3}:\qquad &\la q_0-a\la q_1+(ac-b)q_2=0.
\end{align*}

 We analyze now the three cases in which $\det H$ vanishes as those points $q$
 form the set where $\pi^{-1}(q)\subset S_{A}$. Recall that the only missing
 condition is $p\bar q=0$.

If $q\in r_1$, i.e.\ $q_2=0$, then the conic is reducible as
   $$
p_2(   (-\la q_1)p_0+((\la-1)q_0+aq_1)p_1+(-cq_0+bq_1)p_2)=0,
   $$
and again as in the proof of Theorem~\ref{nostro}, the surface contains the line $L_q$ only if 
$$
\rk\left(\begin{array}{ccc}
  \bar  q_0&\bar q_1&0\\
  0&0&1
 \end{array}\right)=1
\quad\mbox{or}\quad
\rk\left(\begin{array}{ccc}
  \bar  q_0&\bar q_1&0\\
   -\la q_1&(\la-1)q_0+aq_1&-cq_0+bq_1
 \end{array}\right)=1.
$$
Clearly the first matrix has always rank equal to 2, so we are left to study the second one.
As the first column is always different from zero, this has rank equal to one if and only if
\begin{equation}\label{system1}
\begin{cases}
(\la-1)|q_0|^2+a\bar q_0q_1+\la|q_1|^2=0\\
\bar q_0(-cq_0+bq_1)=0.\\
\end{cases}
\end{equation}
It is clear that, as $\la\ne0,1$, then $q_0=0$ if and only if $q_1=0$, but as
we are in the case in which $q_2=0$, this option is not possible. Hence we have
$bq_1=cq_0$.  We have that $b=0$ if and only if $c=0$. Hence, if $b\ne0\ne c$
we obtain a unique solution whenever $b$ and $c$ are such that
\begin{equation}\label{eqbc}
-(1-\la)|b|^2+\la|c|^2+a\bar b c=0.
\end{equation}

Assume now that {$b=c=0$}, then we are left to deal with the equation
$(\la-1)|q_0|^2+a\bar q_0q_1+\la|q_1|^2=0$. By setting $z=q_1/q_0$, the last
equation is equivalent to
\begin{equation}\label{equation1}
|z|^2+\frac{a}{\la}z+\frac{\la-1}{\la}=0.
\end{equation}


{Now, if $a\ne0$, }it is sufficient to apply Corollary \ref{zef}: in this case
we have
$$\De=\tfrac1{|\la|^4}\left(|a|^4-4|a|^2(|\la|^2-\la_0)-4\la_1^2)\right)
$$

If $\De>0$, then we have two solutions, while if $\De=0$, then we have
a unique solution.

If $q\in r_{2}$ we obtain
\begin{equation}\label{eqac}
|a|^2(1-\la)+|c|^2{+} a \bar b c=0,
\end{equation}
while, if $q\in r_{3}$ we get
\begin{equation}\label{eqab}
|a|^2\la +|b|^2-a\bar b c=0.
\end{equation}
Now, the analysis of these two conditions is completely analogous to the first case and we find cases
$(2), (5), (8)$ or $(3), (6), (9)$, respectively.

Finally, as the conditions $a=0=b$, $a=0=c$ and $b=0=c$ are mutually exclusive, we only need to check that
the twistor line obtained in the case $a\neq0$, $b\neq0$ and $c\neq 0$ is at most one.
This means that no two of the conditions given in Formulas~\eqref{eqbc}, \eqref{eqac} and~\eqref{eqab} 
can coexist,
i.e., the following systems have no solution $[a:b:c]\in\CP$:
$$
\begin{cases}
(1-\la)|b|^2-\la|c|^2-a\bar b c=0\\
-|a|^2(1-\la)-|c|^2{-} a \bar b c=0,
\end{cases}\quad
\begin{cases}
(1-\la)|b|^2-\la|c|^2-a\bar b c=0\\
|a|^2\la +|b|^2-a\bar b c=0,
\end{cases}\quad
\begin{cases}
-|a|^2(1-\la)-|c|^2{-} a \bar b c=0\\
|a|^2\la +|b|^2-a\bar b c=0.
\end{cases}\quad
$$
But these three systems are equivalent to
$$
\begin{cases}
(1-\la)|b|^2-\la|c|^2-a\bar b c=0\\
(|a|^2+|b|^2+|c|^2)(1-\la)=0,
\end{cases}\quad
\begin{cases}
(1-\la)|b|^2-\la|c|^2-a\bar b c=0\\
\la(|a|^2 +|b|^2+|c|^2)=0,
\end{cases}\quad
\begin{cases}
-|a|^2(1-\la)-|c|^2{-} a \bar b c=0\\
|a|^2 +|b|^2+|c|^2=0,
\end{cases}\quad
$$
respectively, and hence, none of these admit a solution $[a:b:c]\in\CP$.
\end{proof}

\begin{remark}
We point out that Example \ref{primo ex} provides an example of a surface
containing two twistor lines (case $(7)$ in the previous Proposition), Example
\ref{secondo ex} provides and example of surface containing one twistor line
(case $(4)$) and Example \ref{terzo ex} gives an example of surface without
twistor lines.
\end{remark}

\subsection{Ramification again}

Theorem \ref{twistorfibers} analysed the locus $\Dd_0=\{z\in\CP\mid
\pi^{-1}(z)\subset S\}$, (cf.\ Definition \ref{def-discriminant}). We know that
it consists of zero, one, or two points, depending on conditions on
$a,b,c,\la$. Now we want to analyse the branch locus of a smooth non
$j$-invariant $(1,1)$-surface.

\begin{theorem}
Let $S_A$ be a smooth non $j$-invariant surface of bidegree $(1,1)$, such
that $A$ is given by \eqref{forma-canonica}. Then the twistor
projection $\pi$ restricted to $S$ is a degree 2 cover of $\CP$ whose branch
locus $\Rr\subset\CP$ is the zero set of
\[\begin{array}{rcl} R(q)
&=& \Big\{|q_0|^2(\la-1)+|q_1|^2\la+|q_2|^2+a\bar q_0q_1+b\bar q_0q_2
+c\bar q_1q_2\Big\}^2\\[2pt]
&& -4\Big\{|q_0|^2|q_2|^2(\la-1)
  +a\big[c\bar q_0q_2(|q_0|^2+|q_1|^2(\la+1)+|q_2|^2)
  -\bar q_0q_1|q_2|^2(\la-1)\big]\\
  && \hskip30pt
  -b\bar q_0q_1|q_1|^2+c\bar q_1q_2(|q_0|^2\la+|q_1|^2\la+|q_2|^2)\Big\}.
\end{array}\]
\end{theorem}

\begin{proof}
Following the proof of Theorem~\ref{nostro}, we have that the branch locus is
given by the set of points $q\in\CP$ such that
\begin{equation}
  \rk\left(\begin{array}{cc}
    \bar q_0 & q_2p_1+(-q_1\la +cq_2)p_2\\
    \bar q_1 & q_2p_0-2aq_2p_1+(q_0(\la-1)+aq_1-bq_2)p_2\\
    \bar q_2 & (-q_1\la +cq_2)p_0+(q_0(\la-1)+aq_1-bq_2)p_1+2(-cq_0+bq_1)p_2\\
  \end{array}\right)=1
\end{equation}
which, together with $\bar q\cd p=0$ gives  the following linear system in $[p_1:p_2:p_{3}]\in\CP$
\begin{equation}\label{sistemone}
\begin{cases}
 (\bar q_0 q_2)p_0+(-2a\bar q_0 q_2-\bar q_1 q_2)p_1+(|q_0|^2(\la-1)+a\bar q_0 q_1-b\bar q_0 q_2+|q_1|^2\la-c\bar q_1 q_2)p_2=0\\
(-\bar q_0 q_1\la +\bar q_0 q_2c)p_0+(|q_0|^2(\la-1)+a\bar q_0q_1-b\bar q_0 q_2-|q_2|^2)p_1\\
\qquad+(-2c|q_0|^2+2b\bar q_0 q_1+\bar q_2 q_1\la-c|q_2|^2)p_2=0\\
(-|q_1|^2\la+\bar q_1q_2c-|q_2|^2)p_0+(\bar q_1 q_0(\la-1) +a| q_1|^2-b\bar q_1 q_2+2a| q_2|^2)p_1\\
\qquad+( -2c\bar q_1 q_0+2b|q_1|^2-q_0\bar q_2(\la-1)-a\bar q_2q_1+b|q_2|^2)p_2=0 \\
\bar q_0p_0+\bar q_1 p_1 + \bar q_2p_2=0.
 \end{cases}
\end{equation}
Such system admits a non-zero solution if and only if the associated matrix 
\def\vpd{\vphantom{\lower6pt\hbox{1}}}
\def\vpu{\vphantom{\raise6pt\hbox{1}}}
\def\entry#1#2{\begin{array}{c}\vpu#1\\#2\vpd\end{array}}
\[ M=\left(\begin{array}{c|c|c}
  \bar q_0 q_2&-2a\bar q_0 q_2-\bar q_1 q_2&
  \entry{|q_0|^2(\la-1)+a\bar q_0 q_1-b\bar q_0 q_2}
    {+|q_1|^2\la-c\bar q_1 q_2}\\\hline
    -\bar q_0 q_1\la +\bar q_0 q_2c&
    \entry{|q_0|^2(\la-1)+a\bar q_0q_1}{-b\bar q_0 q_2-|q_2|^2}
    & \entry{-2c|q_0|^2+2b\bar q_0 q_1}{+\bar q_2 q_1\la-c|q_2|^2}\\\hline
    -|q_1|^2\la+\bar q_1q_2c-|q_2|^2&
    \entry{\bar q_1 q_0(\la-1) +a| q_1|^2}{-b\bar q_1 q_2+2a| q_2|^2}
    & \entry{-2c\bar q_1 q_0+2b|q_1|^2-q_0\bar q_2(\la-1)}
          {-a\bar q_2q_1+b|q_2|^2}\\\hline
        \bar q_0&\bar q_1  & \bar q_2
  \end{array}\right)\]
has rank strictly less than three. This happens if and only if the determinants
of the $3\times3$ submatrices all vanish.

If we denote by $M_i$ the determinant of the submatrix obtained from $M$ by
removing the $i$th row, then we have
$$
M_1=-\bar q_2R,\quad M_2=-\bar q_1R,\quad M_{3}=-\bar q_0R,\quad M_4=0,
$$
where $R=R(q)$ is as stated.
\end{proof}

\begin{remark}
Observe that if $a=b=c=0$ and $\la\in\CC\sm\RR$ then $R(q)=0$
corresponds to the equations obtained in Theorems~\ref{nostro}
and~\ref{diagonaltheorem}.
\end{remark}

We shall conclude by investigating a case in which (to simplify matters)
  $b=c=0$, and both $a$ and $\la$ are real. In this case $S_A$ is invariant by
the action of $U(1)$ on $p_2$ and $\ell_2$ in $(p,\ell)$. This residual
symmetry enables us to exhibit the branch locus as an explicit surface of
revolution in $\RR^3$.
  
 \begin{corollary}\label{a2s}
  Suppose that $A$ is given by \eqref{forma-canonica} with $a\in\RR$,
  $a\ne0$, $b=c=0$, and $\la=2$. Then the branch locus of $\pi\colon
  S_A\to\CP$ is homeomorphic to a torus with $0$, $1$ or $2$ singular
  points.
\end{corollary}

\begin{proof}
We are assuming that $b=c=0$ and $\la=2$. Set $q_1=x+iy$ and
$q_2=u+iv$. If $q_0=0$ then
$$R(q)  = \left(2(x^2+y^2)^2+u^2+v^2\right)^2,$$
and the zero set is empty in $\CP$. So we can set $q_0=1$ and use $x,y,u,v$ as
real inhomogeneous coordinates.
We set $r=|q_{1}|$ and $s=|q_{2}|$. It follows that
\[\Im R(q) = 2ay\big(1+ax+2r^{2}+3s^{2}\big),\]
and 
\[\Re R(q) = (1+2r^2+s^2+ax)^2-a^2y^2-4s^2+4axs^2.\]
To study $R(q)=0$, we first assume $y\ne0$. The imaginary part being zero
implies that $ax<0$ and
\[\Re R(q) = (-2s^2)^2-a^2y^2-4s^2+4axs^2= 4s^4-4(1-ax)s^2-a^2y^2.\]
Now, this vanishes if and only if 
\[\ts 2s^2 = (1-ax)+\sqrt{(1-ax)^{2}+a^{2}y^{2}}.\]
Going back to the imaginary part, we obtain that
\[\ts \Im R(q)=1+ax+2r^{2}+\frac32((1-ax)+\sqrt{(1-ax)^{2}+a^{2}y^{2}}) =
\frac{5}{2}-\frac{1}{2}ax+2r^{2}+\frac{3}{2}\sqrt{(1-ax)^{2}+a^{2}y^{2}},\]
which is strictly positive and hence there are no solutions.

Setting $y=0$, gives $\Im R(q)=0$ and we are reduced to study
\begin{equation}\label{r4r2}
  \Re R(q) = s^4+(4x^2+6ax-2)s^2+(1+ax+2x^2)^2=0.
\end{equation}
This defines a doubled lima\c con-shaped curve in the $(x,s)$ plane. The
resulting surface of revolution is formed by rotating this profile around the
$x$-axis. If $s=0$ then $2x=-a\pm\sqrt{a^2-8}$. It follows that the surface is
smooth if $0<|a|<2\sqrt2$, has one singular point if $|a|=2\sqrt2$, and two if
$|a|>2\sqrt2$. The torus is not circular because we do not have the full $T^2$
symmetry.
\begin{figure}
  \includegraphics[scale=1.1]{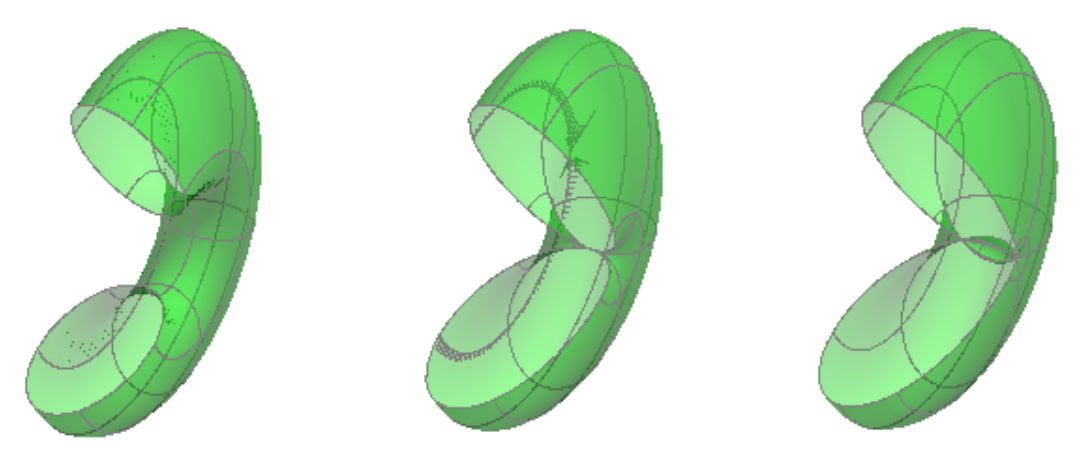}
  \vspace{10pt}
\caption{Part of the branch loci of $S_{2;a,0,0}$ for $a=1$, $a=2\sqrt2$, $a=6$}
\label{disc}
\end{figure}
\end{proof} 

\begin{remark}
  As it stands, Corollary \ref{a2s} is consistent with Theorem
  \ref{twistorfibers}, and with the examples contained in
  Section~\ref{configuration-sec}. If, on the other hand, $a=0$, we
  are back in the situation of a real $(1,1)$ surface, see Subsection
  \ref{real11}. The choice $\la=2$ places us in the second line of
  \eqref{qqq} with $0=q_1=x+iy$. This is completely consistent with
  \eqref{r4r2}, which reduces to
\[ s^4+2(2x^2-1)s^2+(2x^2+1)^2=0,\]
which has a unique solution $x=0$ and $s=1$. After rotation, this is a
circle in the $q_2=u+iv$ plane.
\end{remark}

Cut-away versions of the associated surfaces are shown in Figure
\ref{disc}. The three values of $a$ correspond to a smooth torus, a
horn torus and a spindle torus respectively
\cite[pp.\ 305--306]{GAS}. These images provide a vivid analogy with
the branch loci of quadrics in $\PP^3$ \cite{sv1}. One might
conjecture that the branch loci of more general $(1,1)$ surfaces in
$\FF$ with finitely many twistor fibers fall topologically into the
three categories illustrated by Figure \ref{disc}, but verification is
beyond the scope of the present paper.

\bigbreak

\noindent{\bf Acknowledgments.} The fourth author is grateful to Fran
Burstall and Nick Shepherd-Barron for explaining relevant facts to him. We thank the referee for many helpful comments.

\vskip20pt

\section{Bibliography}\vskip-20pt

\addtocontents{toc}{\protect\setcounter{tocdepth}{-1}}

\enddocument